\documentclass[11pt, dvipsnames]{article}
\usepackage[utf8]{inputenc}
\usepackage{amsmath}
\usepackage[english]{babel}
\usepackage{amsmath}
\usepackage{amssymb}
\usepackage{amsthm}
\usepackage{mathtools}
\usepackage{verbatim}
\usepackage{enumerate}
\usepackage[linesnumbered,boxruled,lined]{algorithm2e}
\usepackage[dvipsnames]{xcolor}
\usepackage{cite}
\usepackage{subfig}
\usepackage{pgfplots}
\usepackage{float}
\usepackage{epstopdf}
\usepackage{etoolbox}
\usepackage{tabularx}
\usepackage{tikz}
\usepgfplotslibrary{fillbetween}
\newcommand{\norm}[1]{\left\Vert #1\right\Vert}
\newcommand{\real}{\mathbb{R}}

\newcommand{\set}[1]{\left\{ #1\right\}}

\newcommand{\trans}[1]{#1^{\mathsf{T}}}
\newcommand{\T}{\mathsf{T}}
\newcommand{\X}{\mathcal{X}}
\newcommand{\Tset}{\mathcal{T}}

\DeclareMathOperator{\Int}{int}
\DeclareMathOperator{\cl}{cl}
\DeclareMathOperator{\bd}{bd}
\DeclareMathOperator{\conv}{conv}
\DeclareMathOperator{\cone}{cone}

\DeclareMathOperator{\vertice}{vert}

\theoremstyle{plain}
\newtheorem{theorem}{Theorem}[section]
\newtheorem{definition}[theorem]{Definition}
\newtheorem{rem1}[theorem]{Remark}
\newtheorem{proposition}[theorem]{Proposition}
\newtheorem{lemma}[theorem]{Lemma}
\newtheorem{assumption}[theorem]{Assumption}

\theoremstyle{definition}
\newtheorem{example}[theorem]{Example}
\newenvironment{remark}{\begin{rem1}\rm}{\end{rem1}}

\usepackage{geometry}
\theoremstyle{definition}
{\rule{\textwidth}{1pt}\begin{alg}[#1]~\\\rule{\textwidth}{1pt}\\}		
	{\end{alg}\rule{\textwidth}{1pt}}				
\newtheorem{alg}[theorem]{Algorithm}

\newcommand{\rev}[1]{\textcolor{black}{#1}}

\newcommand{\g}[1]{\begin{color}{gray}{#1}\end{color}}

\newrobustcmd*{\csquare}[1]{\tikz{\filldraw[draw=#1,fill=#1] (0,0)		rectangle (0.1cm,0.1cm);}}

\newrobustcmd*{\ctriangle}[1]{\tikz{\filldraw[draw=#1,fill=#1] (0,0) --		(0.15cm,0) -- (0.075cm,0.15cm);}}

\begin{document}
	\title{Algorithms to solve unbounded convex vector optimization problems}
	\author{Andrea Wagner\thanks{Vienna University of Economics and Business, Institute for Statistics and Mathematics, Vienna A-1020, AUT,  andrea.wagner@wu.ac.at, birgit.rudloff@wu.ac.at, gabriela.kovacova@wu.ac.at, nihey@wu.ac.at.} \and Firdevs Ulus\thanks{Bilkent University, Department of Industrial Engineering, Ankara, 06800 Turkey, firdevs@bilkent.edu.tr.} \and Birgit Rudloff\footnotemark[1] \and Gabriela Kov\'{a}\v{c}ov\'{a}\footnotemark[1]  \and Niklas Hey\footnotemark[1]
	}
	\maketitle
	\abstract{
		This paper is concerned with solution algorithms for general convex vector optimization problems (CVOPs). So far, solution concepts and approximation algorithms for solving CVOPs exist only for bounded problems~\cite{LRU14,DorLohSchWei2021,AraUlusUmer2021}. They provide a polyhedral inner and outer approximation of the upper image 
		that have a Hausdorff distance of at most $\varepsilon$. However, it is well known (see~\cite{Ulus18}), that for some unbounded problems such polyhedral approximations do not exist. In this paper, we will propose a generalized solution concept, called an $(\varepsilon,\delta)$--solution, that allows also to consider unbounded CVOPs. It is based on additionally bounding the recession cones of the inner and outer polyhedral approximations of the upper image in a meaningful way. An algorithm is proposed that computes such 
		$\delta$--outer and $\delta$--inner approximations of the recession cone of the upper image. In combination with the results of~\cite{LRU14} this provides a primal and a dual algorithm that allow to compute $(\varepsilon,\delta)$--solutions of (potentially unbounded) CVOPs. Numerical examples are provided.
		\\[.2cm]
		\\[.2cm]
	}

\section{Introduction}
Minimizing a vector valued objective function with respect to a partial order relation
induced by a cone over a feasible region is referred to as vector optimization. The subject of this paper is the special case of convex problems, in which the objective function is convex with respect to the given order relation and the feasible region is convex. There are many application areas in which solving convex vector optimization problems (CVOPs) plays a key role. Some of the recent application areas can be listed as financial mathematics \cite{FeiRud2017}, economics \cite{RudUlu2021} and dynamic programming \cite{RudKov2021}. Moreover, recently in \cite{DeSanEicNieRoc2020}, convex vector optimization algorithms are used in order to solve convex relaxations of nonconvex problems. Convex projections~\cite{KovacovaRudloff2021} can also be solved via multi-objective convex problems. 

Solving a CVOP has different meanings depending on the solution concept that one considers. The main idea is to find finitely many efficient solutions which generate a good representative of the whole set of efficient values in the objective space. Motivated from the set optimization point of view, an exact solution concept for linear vector optimization problems is introduced by L\"ohne \cite{Loe11}. Accordingly, a solution is a set of efficient solutions whose images generate the so called \emph{upper image}, which can be described simply as the image of the feasible set added to a predefined ordering cone. Since the upper image is a polyhedral set in the linear case, it is possible to generate the finitely many extreme points and extreme directions\rev{, see for instance \cite{Ben98,EhrLoeSha12,HLR13,LoeWei16,RudUlVan17} for algorithms that generate such solutions}. This clearly may not be possible for CVOPs\rev{, hence one often considers approximate solutions for CVOPs. Indeed, algorithms to compute polyhedral approximations to the upper image of convex multiobjective optimization problems are proposed in \cite{KlaTiWie2003,EhrShaSch11}. Later,} an $\varepsilon$-solution concept \rev{for CVOPs} is defined in \cite{LRU14}. Accordingly, a finite (weak) $\varepsilon$-solution consists of finitely many (weakly) efficient solutions whose images generate a polyhedral inner approximation, say $P$, to the upper image such that shifting $P$ through a fixed direction at the size of $\varepsilon$ yields an outer approximation to the upper image. \rev{In \cite{KesUlus2022}, a framework for the objective space-based CVOP algorithms is provided and different variants of the primal algorithm from \cite{LRU14} that guarantee finding (weak) $\varepsilon$-solutions are proposed and compared.} Recently, in \cite{DorLohSchWei2021} and \cite{AraUlusUmer2021}, a similar solution concept which does not depend on a fixed direction is defined. Similar to the one from \cite{LRU14}, these solution concepts guarantee that the Hausdorff distance between the upper image and the corresponding polyhedral approximation is bounded by~$\varepsilon$.   

The algorithms from \cite{LRU14,DorLohSchWei2021,AraUlusUmer2021} that can be used to compute such \rev{approximate} solutions assume that the feasible region of the CVOP is compact. Under this assumption, it is known that the problem is \emph{bounded} (the upper image is a subset of a point plus the ordering cone) and the existence of an $\varepsilon$-solution is guaranteed, see for instance \cite[Proposition 4.2]{LRU14}. 
However, for many applications, the compactness of the feasible region is a restrictive assumption, see for instance the introduction of \cite{KovacovaRudloff2021}.

Note that if the problem does not have a compact feasible region, then the problem may be \emph{unbounded}. In \cite{Ulus18}, it is shown that if the problem is unbounded, there may not exist a polyhedral approximation to the upper image such that the Hausdorff distance between the two is finite. This means that the solution concepts from \cite{LRU14,DorLohSchWei2021,AraUlusUmer2021} may not be applicable for such problems. Indeed, it has been shown in \cite{Ulus18} that finding a polyhedral outer approximation with bounded Hausdorff distance may be possible only if the problem is \emph{self-bounded} (the upper image is a subset of a point plus its recession cone). There is also a solution concept for such problems in~\cite{Ulus18}. However, it is more a theoretical concept rather than a practical one because it requires computing the recession cone of the upper image, which even may not be a polyhedral set in general. 

Motivated by this, we focus in this paper on general convex vector optimization problems without an assumption of boundedness. We propose a\rev{n approximate} solution concept -- an $(\varepsilon,\delta)$--solution -- which is applicable for bounded as well as for unbounded CVOPs. This new solution concept consists of two parts, a set of weak minimizers to approximate the central part of the Pareto frontier and a set of directions to approximate the recession cone of the upper image. Jointly, these two components provide a (polyhedral) approximation of the upper image. 

We further provide a method for finding $(\varepsilon,\delta)$--solutions of CVOPs. \rev{Similar to the approach for solving linear vector optimization problems, see for instance \cite{bensolve,Loe14}, our} method consists of two phases: In the first phase, which is represented by Algorithm~\ref{alg_new}, a $\delta$--approximation of the recession cone of the upper image is found. In the second phase, we use the newly found (outer) approximation of the recession cone to transform the CVOP into a bounded problem and proceed with searching for weak minimizers to approximate the central part of the frontier. For this purpose, available algorithms for bounded CVOPs can be adapted. We provide both, a primal algorithm (see Algorithm~\ref{alg_2}) and a dual algorithm (see Algorithm~\ref{alg_dual}), based on combining Algorithm~\ref{alg_new} with the primal and dual algorithms of~\cite{LRU14}. We illustrate our algorithms with numerical examples.

The paper is organized as follows: Section 2 provides basic definitions and notations. In Section~\ref{sec_problem_formulation}, the definition and properties of convex vector optimization problems as well as solution concepts for bounded CVOPs are recalled. Then, a solution concept for unbounded problems is proposed and a transformation of an unbounded CVOP into a bounded one is given that is using an outer approximation of the recession cone of the upper image as the ordering cone. Section~\ref{sec_scalar} recalls scalarizations of CVOPs that will be used within subsequent algorithms and provides related results. Section~\ref{sec_out_appr} contains the main contribution of this paper -- it provides a method to approximate the recession cone of an upper image of a CVOP. Section~\ref{sec_algs} extends this into algorithms computing $(\varepsilon,\delta)$--solutions of CVOPs, it contains both a primal and a dual variant of the method. \rev{We provide examples in Section~\ref{sec_ex} and conclude the paper in Section~\ref{sec:conc}.}

\section{Preliminaries}
Standard basis vectors of $\real^q$ are denoted by $e^1, \dots, e^q$.
For a set $A \subseteq \mathbb R^q$ we denote the interior, closure, boundary, convex hull and convex conic hull of $A$ by $\Int A$, $\cl A$, $\bd A$, $\conv A$ and $\cone A$, respectively. The \textit{recession cone} $A_{\infty}$ of a set $A \subseteq \mathbb R^q$ is defined as $A_{\infty}:=\{d \in \mathbb R^q \mid \forall y \in A, \forall \alpha \geq 0: y+\alpha d \in A\}$. A {\em{polyhedral convex}} set $A \subseteq \mathbb R^q$ can be expressed as 
\begin{align}{\label{H-rep}}
	A = \bigcap_{i=1}^r \left\lbrace y \in \mathbb R^q \mid \trans{(w^i)}y \geq \gamma_i\right\rbrace
\end{align}
for some $r \in \mathbb{N}, w^1,...,w^r \in \mathbb R^q \setminus \{0\}$ and $\gamma_1,...,\gamma_r \in \mathbb{R}$. Representation (\ref{H-rep}) is called {\em{H-representation}} of a convex polyhedron. A convex polyhedron $A$ can also be represented as
\begin{align}{\label{V-rep}}
	A = \conv \left\lbrace x^1,...,x^s \right\rbrace +\cone \left\lbrace k^1,...,k^t \right\rbrace,
\end{align}
where $s \in \mathbb N \setminus \{0\}$, $t \in \mathbb N$, $x^1,...,x^s \in \mathbb R^q$ are points (vertices) in $A$ and $k^1,...,k^t \subseteq \mathbb R^q \setminus \{0\}$ are directions of $A$. The set of points $\left\lbrace x^1,...,x^s \right\rbrace$ together with the set of directions $\left\lbrace k^1,...,k^t \right\rbrace$ are called {\em{generators}} of the convex polyhedron $A$. Representation (\ref{V-rep}) is called {\em{V-representation}} or {\em{generator representation}} of a convex polyhedron. Note that we use the convention $\cone \emptyset=\{0\}$. If the set of directions in the generator representation (\ref{V-rep}) is empty, we call $A$ a \textit{polytope}. We denote by $\vertice A$ the set of vertices of a polytope $A$. 

Recall that a set $C \subseteq \mathbb R^q$ is a {\em{cone}} if $\alpha C \subseteq C$ holds for all \rev{$\alpha\geq0$}.
A nonzero element $c \in C$ is an extreme direction of $C$ if it cannot be expressed as a convex combination of elements of $C \setminus \cone \{c\}$.
The {\em{positive dual cone}} of a cone $C \subseteq \mathbb R^q$ is the set $C^{+}:=\{z \in \mathbb R^q \mid \trans{z}c \geq 0, \forall c \in C\}$. A cone $C$ is {\em{convex}} if $C + C \subseteq C$ holds. A convex cone $C \subseteq \mathbb R^q$ is said to be {\em{pointed}} if it contains no lines through the origin. If $\Int C \neq \emptyset$, then the convex cone $C \subseteq \mathbb R^q$ is called {\em{solid}}. A cone $C \subseteq \mathbb R^q$ is called \textit{non-trivial} if $\{ 0 \} \subsetneq C \subsetneq \mathbb{R}^q$. A non-trivial pointed convex cone $C \subseteq \mathbb R^q$ defines a partial ordering $\leq_C$ in the space $\mathbb R^q$: $v \leq_C w$ if and only if $w-v \in C$ for $v,w \in \mathbb R^q$. If $C$ is additionally solid, we denote by $v <_C w$ that $w-v \in \Int C$ for $v,w \in \mathbb R^q$. 

Let $C \subseteq \mathbb R^q$ be a non-trivial pointed convex cone and let $\mathcal{X} \subseteq \mathbb R^n$ be a convex set. A vector-valued function $\Gamma: \mathcal{X} \to \mathbb R^q$ is said to be {\em{C-convex}} if for all $x,y \in \mathcal{X}$ and for all $\alpha \in [0,1]$ it holds $\Gamma(\alpha x+(1-\alpha)y) \leq_C \alpha \Gamma(x)+(1-\alpha)\Gamma(y)$. Moreover, for a nonempty set $A \subseteq \mathbb R^q$, $y \in A$ is called a {\em{C-minimal element}} of $A$ if $\left(\{y\}-C\setminus \{0\} \right) \cap A=\emptyset$. If the cone $C$ is solid, then a point $y \in A$ is called a {\em{weakly C-minimal element}} of $A$ if $\left( \{y\}-\Int C \right)\cap A=\emptyset$. 

Throughout, we use the $\ell_1$ (Manhattan) norm $\norm{ \cdot } = \norm{ \cdot }_1$. A closed unit ball around the origin is denoted by $B_1 (0) = \{ x \in \mathbb{R}^q \mid \norm{x} \leq 1 \}$. The {\em{Hausdorff-distance}} (under the $\ell_1$ norm) between two nonempty sets $A,B \subseteq \mathbb R^q$ is defined as 
\begin{align*}
	\text{d}_H(A,B):=\max \left\{ \sup_{a \in A} \inf_{b \in B} \norm{a-b},\sup_{b \in B} \inf_{a \in A} \norm{a-b} \right\}.
\end{align*}
The $\ell_2$ (Euclidean) norm, when needed, is denoted by $\norm{ \cdot }_2$. Recall that the two norms are equivalent, in particular it holds $\norm{ x }_2 \leq \norm{ x }_1 \leq \sqrt{q} \norm{ x }_2$ for all $x \in \mathbb{R}^q$.

\section{Problem formulation and solution concept}
\label{sec_problem_formulation}
A convex vector optimization problem (CVOP) is an optimization problem of the form
\begin{align*}
	\tag{P} \label{eq:P}  \text{minimize } \Gamma(x) \quad \text{ with respect to\ } \leq_C\quad\text{ subject to } x\in\mathcal{X},
\end{align*}
where $C \subseteq \mathbb R^q$ is a non-trivial pointed \rev{solid} convex ordering cone, $\mathcal{X} \subseteq \mathbb R^n$ is a convex set and $\Gamma: \mathcal{X} \to \mathbb R^q$ is a $C$-convex function. The set $\mathcal{X}$ is assumed to be non-empty, i.e. \eqref{eq:P} is feasible. The {\em{image of the feasible set}} $\mathcal{X}$ is the set $\Gamma(\mathcal{X}):=\{\Gamma(x) \mid x \in \mathcal{X}\}$ and the {\em{upper image}} of \eqref{eq:P} is defined by 
$\mathcal{P}:=\cl \left( \Gamma(\mathcal{X})+C \right).$

Throughout the rest of the paper we make the following assumptions.
\begin{assumption} \label{assumption1}
	Let the following hold true:
	\begin{enumerate}[(a)]
		\item The feasible set $\mathcal{X} \subseteq \mathbb R^n$ is closed, convex and $\Int \mathcal{X} \neq \emptyset$.
		\item The ordering cone $C \subseteq \mathbb R^q$ is closed in addition to being non-trivial, pointed, solid and convex.
		\item The objective function $\Gamma: \; \mathcal{X} \to \real^q$ is continuous and $C$-convex.
		\rev{\item A direction $c \in \Int C$ with $\norm{c} = 1$ is fixed throughout the paper. }
	\end{enumerate}
	\rev{
		For computational purposes (from Section~\ref{sec_out_appr} onward) we further assume:
		\begin{enumerate}[(e)]
			\item The ordering cone $C \subseteq \mathbb R^q$ is polyhedral.
		\end{enumerate}		
	}
\end{assumption}

\rev{Note that Assumptions~\ref{assumption1} (a)-(c) are standard for numerical methods/algorithms for CVOPs. Assumption~\ref{assumption1} (d) is needed for the solution concept to be discussed in Section~\ref{subsec:solncon} as well as for the algorithms that will be presented in Section~\ref{sec_algs}. The results of this paper hold true with any choice of $c$ satisfying Assumption~\ref{assumption1} (d). However, the choice of $c$ may affect the computational performance of the algorithms. Indeed, this has been shown recently in \cite{KesUlus2022} for \cite[Algorithm 1]{LRU14}. 
}

\begin{definition}[see~\cite{Ulus18}]
	Let $\hat{C} \supseteq C$ be a closed convex cone. Problem \eqref{eq:P} is called \textbf{bounded with respect to~$\mathbf{\hat{C}}$} if there is a point $\hat{p} \in \mathbb R^q$ such that $\mathcal{P} \subseteq \set{\hat{p}}+\hat{C}$. Problem \eqref{eq:P} is called \textbf{bounded} if it is bounded with respect \rev{to} the ordering cone $C$.  Problem \eqref{eq:P} is called \textbf{unbounded} if it is not bounded. If \eqref{eq:P} is bounded with respect to~$\mathcal{P_{\infty}}$, then it is called \textbf{self-bounded}.
\end{definition}

Since $C \subseteq \mathcal{P_{\infty}}$ always holds true, a bounded problem is also self-bounded. On the other hand, a self-bounded problem does not necessarily need to be bounded. We refer the reader to~\cite{Ulus18} for more details and examples.

\rev{
	Note that a vector optimization problem with a polyhedral ordering cone can be transformed into a multi-objective problem (with the standard ordering cone), we refer reader to~\cite{SawNakTan85} for more information. However, if the original vector optimization problem is unbounded, then also the resulting multi-objective problem is unbounded. Computationally, such a transformation is of advantage if it reduces the dimension of the objective space.
}

\subsection{Solution concept}\label{subsec:solncon}
We will now establish an appropriate solution concept for problem~\eqref{eq:P}.  To do so, let us introduce the basic \rev{notations} and recall existing solution concepts for bounded and self-bounded problems from the literature.
A point $\bar{x} \in \mathcal{X}$ is called a {\em{(weak) minimizer}} for \eqref{eq:P} if $\Gamma(\bar{x})$ is a (weakly) $C$-minimal element of $\Gamma(\mathcal{X})$. For expository purposes we will occasionally refer to \rev{the} set of all images of (weak) minimizers for \eqref{eq:P} as the Pareto frontier, where we will not explicitly differentiate between weak minimizers and minimizers as the purely expository statements involving the Pareto frontier will be true for both.

A solution concept for a bounded problem \eqref{eq:P} is given in \cite{LRU14}. This solution concept\rev{, which uses the fixed direction $c \in \Int C$,} is based on the complete lattice approach in set optimization (see~\cite{HeyLoe11}) and it has to fulfill two conditions: infimum attainment and minimality. Since it is not possible in general to represent the upper image $\mathcal{P}$ of a convex vector optimization problem by finitely many points, a polyhedral approximation is sought.
Given a tolerance $\varepsilon>0$, a \textit{finite (weak) $\varepsilon$--solution} $\bar{\X} \subseteq \mathcal{X}$ of a bounded problem \eqref{eq:P} consists of finitely many (weak) minimizers  
and satisfies
\begin{align}
	\label{sol-concept_bd}
	\mathcal{P} \subseteq \conv \Gamma(\bar{\X}) + C - \varepsilon\{c\}.
\end{align}
Thus, one obtains 
a polyhedral inner and outer approximation \rev{of the upper image} in $\varepsilon$--distance to each other that contain the Pareto frontier in between, that is
\begin{align}
	\label{sol-concept_bd_io}
	\conv \Gamma(\bar{\X}) + C\subseteq \mathcal{P} \subseteq \conv \Gamma(\bar{\X}) + C - \varepsilon\{c\}.
\end{align}
This solution concept has been extended to self-bounded problems in~\cite[Definition~5.1]{Ulus18}, where the cone $C$ in~\eqref{sol-concept_bd} (and thus also in~\eqref{sol-concept_bd_io}) is replaced by the recession cone $\mathcal{P_{\infty}}$ of the upper image $\mathcal{P}$. Unless $\mathcal{P_{\infty}}$ is polyhedral and known, this is more of a theoretical concept, as until now no method was known on how to compute or approximately calculate $\mathcal{P_{\infty}}$. This problem will be resolved in the present paper. The situation for unbounded problems (that are not self-bounded) is even worse, as it is known that it is not possible  in this case to generate polyhedral inner and outer approximations \rev{of the upper image} in $\varepsilon$--distance that contain the Pareto frontier, see~\cite{Ulus18}, and thus no solution concept existed so far for this type of problem.

In the following, we will propose a solution concept for general convex vector optimization problems~\eqref{eq:P}. We will see that \rev{introducing} an  $\delta$--approximation of $\mathcal{P}_{\infty}$ will resolve the issues for both the self-bounded as well as the unbounded case. In Section~\ref{sec_out_appr} we will present an algorithm that computes such a $\delta$--approximation. \rev{For a problem~\eqref{eq:P} that is bounded, the algorithm simply returns the ordering cone, whereas for an unbounded problem,} this $\delta$--approximation of $\mathcal{P}_{\infty}$ will include directions that are not contained within the ordering cone $C$.

\begin{definition}{\label{def_approxi_reccP}}
	A (finite) set $\mathcal{Y} \subseteq \mathbb R^q$ is called a \textbf{(finite) $\delta$--outer approximation} of $\mathcal{P}_{\infty}$ if
	\begin{itemize}
		\item the cone generated by $\mathcal{Y}$ is an outer approximation \rev{of $\mathcal{P}_{\infty}$}, i.e. $\mathcal{P}_{\infty} \subseteq \cone \mathcal{Y}$, and
		\item the approximation has at most $\delta$--distance within the unit ball, i.e.
		\begin{align}
			\label{eq_def_1}
			\text{d}_H \left( \mathcal{P}_{\infty} \cap B_1(0), \cone \mathcal{Y} \cap B_1(0) \right) \leq \delta.
		\end{align}
	\end{itemize}
\end{definition}

Note that the intersection with the unit ball in~\eqref{eq_def_1} is necessary for the Hausdorff-distance to be finite; without bounding the sets by the unit ball, the Hausdorff-distance would be infinite whenever the cones $\mathcal{P}_{\infty}$ and $ \cone \mathcal{Y}$ do not coincide.


\begin{definition}{\label{solutionconcept}}
	A pair $(\bar{\X},\mathcal{Y})$ is a \textbf{(weak) $(\varepsilon,\delta)$--solution} of~\eqref{eq:P} if $\bar{\X} \neq \emptyset$ is a set of (weak) minimizers, $\mathcal{Y}$ is a $\delta$--outer approximation of $\mathcal{P}_{\infty}$ and it holds
	\begin{align*}
		\mathcal{P} \subseteq \conv \Gamma (\bar{\X})+\cone \mathcal{Y}-\varepsilon\{c\}.
	\end{align*}
	A (weak) $(\varepsilon,\delta)$--solution $(\bar{\X},\mathcal{Y})$ of~\eqref{eq:P} is a \textbf{finite (weak) $(\varepsilon,\delta)$--solution} of~\eqref{eq:P} if the sets  $\bar{\X},\mathcal{Y}$ consist of finitely many elements.
\end{definition}


In contrast to the solution concept of~\cite{LRU14} for bounded problems given in~\eqref{sol-concept_bd}, Definition~\ref{solutionconcept} contains a set of directions $\mathcal{Y}$ to approximate $\mathcal{P}_{\infty}$. For a bounded~\eqref{eq:P} the choice of $\mathcal{Y} = C$ (or its generators) provides a $0$--outer approximation of $\mathcal{P}_{\infty}$. Thus, if $\bar{\mathcal{X}}$ is an $\varepsilon$--solution in the sense of~\cite{LRU14}, then the pair $(\bar{\X},\mathcal{Y})$ is an $(\varepsilon,0)$--solution in the sense of Definition~\ref{solutionconcept}.
Similarly, if the problem is self-bounded and if  $\mathcal{P_{\infty}}$ is polyhedral and its generators $\mathcal{Y}$ are known, then an $\varepsilon$--solution $\bar{\X}$ in the sense of~\cite{Ulus18} implies that the pair $(\bar{\X},\mathcal{Y})$ is an $(\varepsilon,0)$--solution in the sense of Definition~\ref{solutionconcept}. In all other cases, it will be necessary to compute, in addition to $\bar{\X}$, also a $\delta$--outer approximation $\mathcal{Y}$ of $\mathcal{P}_{\infty}$ in order to solve problem~\eqref{eq:P}.
In particular, this means that the solution concepts provided in~\cite{LRU14,Ulus18} are special cases of the more general solution concept proposed in Definition~\ref{solutionconcept}.

A similar concept was developed independently in~\cite{Doerfler22} in the context of approximating unbounded convex sets by polyhedra. There, the concept of a (polyhedral outer) $(\varepsilon, \delta)$-approximation is introduced, where the tolerance $\delta$ controls the recession directions similarly to~\eqref{eq_def_1} and the tolerance $\varepsilon$ controls the distance of the vertices of the outer approximation to the original set.

\subsection{Approximating the upper image}
Let us now discuss the analog to relation~\eqref{sol-concept_bd_io} for unbounded problems. Recall that in the bounded case one finds in~\eqref{sol-concept_bd_io} a polyhedral inner and outer approximation of the upper image $\mathcal{P}$ that is based on an $\varepsilon$--shift in a direction $c$. Thus, the 
polyhedral inner and outer approximation \rev{of the upper image} as well as the upper image $\mathcal{P}$ \rev{itself} have the same recession cone. In the unbounded case one is still interested in obtaining a polyhedral inner and outer approximation \rev{of $\mathcal{P}$} containing the Pareto frontier in between. We will now discuss how it can be ensured that these approximations are close enough to each other to bound the error of the approximation. In order to do that we define analogously to Definition~\ref{def_approxi_reccP} the following.
\begin{definition}\label{solutionconcept2}
	A (finite) set $\mathcal{Y} \subseteq \mathbb R^q$ is called  a \textbf{(finite) $\delta$--inner approximation} of $\mathcal{P}_\infty$ if $\cone \mathcal{Y} \subseteq \mathcal{P}_\infty$ and~\eqref{eq_def_1} holds. 
\end{definition}
The algorithm provided in Section~\ref{sec_out_appr} computes simultaneously a finite $\delta$--inner approximation as well as a finite $\delta$--outer approximation ($\mathcal{Y}_{In}$ and $\mathcal{Y}_{Out}$) of the recession cone $\mathcal{P}_\infty$, see~Theorem~\ref{thm_alg_1}, that is  $\cone \mathcal{Y}_{In} \subseteq \mathcal{P}_\infty\subseteq \cone \mathcal{Y}_{Out}$. These lead in turn to a polyhedral inner and outer approximation of the upper image
\[
\conv \Gamma(\bar{\X})+\cone \mathcal{Y}_{In}\subseteq\mathcal{P} \subseteq \conv \Gamma(\bar{\X}) +\cone \mathcal{Y}_{Out}-\varepsilon\{c\},
\]
which is the unbounded version of relation~\eqref{sol-concept_bd_io}.
Here, the three recession cones $\cone \mathcal{Y}_{In}$, $\cone \mathcal{Y}_{Out}$, and  $\mathcal{P}_\infty$ of the polyhedral inner and outer approximations and the upper image can differ, but only with at most $\delta$--distance to $\mathcal{P}_\infty$ in the sense of~\eqref{eq_def_1}.
Thus, the error level $\delta$ controls the error of the approximation in terms of an error in the recession cones, which bounds the angle between the approximated recession cone and the true recession cone $\mathcal{P}_{\infty}$, see Remark~\ref{remark_angles} below. And the error level $\varepsilon$ controls the error of the approximation around the Pareto frontier in terms of an $\varepsilon$--shift in a direction $c$ as in the classical solution concept~\eqref{sol-concept_bd_io} up to the point of the frontier where the $\delta$ approximation of the recession cones take over.
Note that for the solution concept in Definition~\ref{solutionconcept} the $\delta$--outer approximation of $\mathcal{P}_\infty$ suffices, as a polyhedral outer approximation of the Pareto frontier suffices to approximately solve the problem.

\begin{remark}
	\label{remark_angles}
	Recall the connection between norm, inner product and angle: For the angle $\alpha$ between vectors $x, y \in \mathbb{R}^q$ it holds $\cos \alpha = \frac{x^\T y}{\|x\|_2 \|y\|_2}$. Given vectors $x, y \in \mathbb{R}^q$ normalized in the $\ell_2$ norm, i.e. $\norm{x}_2 = \norm{y}_2 = 1$, we have $1 - \cos \alpha \leq \frac{\epsilon^2}{2}$ if and only if $\norm{x - y}_2 \leq \epsilon$. Since the $\ell_1$ and $\ell_2$ norms are equivalent, for the $\ell_1$ norm we obtain a similar result. Given vectors $x, y \in \mathbb{R}^q$ with $\norm{x}_1 = \norm{y}_1 = 1$, it holds $1 - \cos \alpha \leq \frac{q \epsilon^2}{2}$  if and only if $\norm{x - y}_1 \leq \epsilon$. This shows that the condition~\eqref{eq_def_1} constrains the angle between the nearest directions of $\cone \mathcal{Y}$ and $\mathcal{P}_{\infty}$.  
\end{remark}

\subsection{Transformation into a bounded problem}
\label{sec_transfo}
We will now discuss how to solve an unbounded problem~\eqref{eq:P} once a $\delta$--outer approximation of the recession cone $\mathcal{P}_{\infty}$ is computed.
The key is to replace the ordering cone in problem~\eqref{eq:P} with the cone generated by the $\delta$--outer approximation of the recession cone $\mathcal{P}_{\infty}$. If this modified problem is bounded, a solution to the original problem~\eqref{eq:P} can be obtained by observing the following.


\begin{remark}\label{remark_definition}
	If a set $\mathcal{Y}$ is a $\delta$--outer approximation of the recession cone $\mathcal{P}_{\infty}$ and problem
	\begin{align}
		\tag{P'} \label{eq:P'}
		\text{minimize } \Gamma(x)\quad \text{ with respect to\ } \leq_{\cone\mathcal{Y}}\quad \text{ subject to } x\in\mathcal{X}
	\end{align}
	is bounded, then~\cite[Algorithm 1]{LRU14} finds a (weak) $\varepsilon$--solution $\bar{\X}$ of~\eqref{eq:P'}, whenever it terminates in finitely many steps. By Definitions~\ref{def_approxi_reccP} and~\ref{solutionconcept}, the pair $(\bar{\X}, \mathcal{Y})$ provides a (weak) $(\varepsilon,\delta)$--solution of problem~\eqref{eq:P}.
\end{remark}

This means that in order to solve problem~\eqref{eq:P}, one needs to find a $\delta$--outer approximation $\mathcal{Y}$ of the recession cone $\mathcal{P}_{\infty}$ such that problem~\eqref{eq:P'} is bounded. An algorithm providing such an approximation $\mathcal{Y}$ (that also guarantees boundedness of~\eqref{eq:P'}) is deduced in Section~\ref{sec_out_appr}, see Theorem~\ref{thm_alg_1}.
Then, the unbounded problem~\eqref{eq:P} with ordering cone $C$ can be transformed into a bounded problem~\eqref{eq:P'}
with ordering cone $\cone\mathcal{Y}$, which can be solved by known algorithms (e.g. \cite{LRU14, DorLohSchWei2021}).
By Remark~\ref{remark_definition},  the solution $\bar{\mathcal{X}}$ of the bounded problem~\eqref{eq:P'}, together with the approximation $\mathcal{Y}$ of $\mathcal{P}_{\infty}$, solve the original problem~\eqref{eq:P} in the sense of Definition~\ref{solutionconcept}.

\section{Scalarizations}
\label{sec_scalar}
The algorithm we introduce in Section~\ref{sec_out_appr} uses two scalarizations of the vector optimization problem~\eqref{eq:P}: the well-known weighted sum scalarization and the Pascoletti-Serafini scalarization. In this section we recall both of them and provide related results.

Let $w \in \mathbb R^q \setminus \{0\}$ be some parameter vector. The following convex program is the {\em{weighted sum scalarization}} of \eqref{eq:P},
\begin{align}
	\tag{P\textsubscript{1}($w$)} 
	\label{P1(w)}
	\text{minimize } \trans{w}\Gamma(x)\quad \text{ subject to } x\in\mathcal{X}.
\end{align}
It is well-known that if $w \in C^{+} \setminus \set{0}$, then an optimal solution $x^w$ of \eqref{P1(w)} is a weak minimizer of \eqref{eq:P} (see e.g. \cite{Jahn04, Luc2011}). The following results will be used later.

\begin{lemma}[Proposition 4.12 and Theorem 4.14 in \cite{Ulus18}]
	\label{lemma_weights}
	The set of weights 
	$$W = \left\lbrace w \in C^+ \; \vert \; \eqref{P1(w)} \text{ is bounded} \right\rbrace$$ 
	is a convex cone for which it holds $\mathcal{P}_{\infty}^+ = \cl W$. If~\eqref{eq:P} is self-bounded, it further holds $\mathcal{P}_{\infty}^+ = W$. Furthermore, if $\{0\} \neq \mathcal{P}_{\infty}^+ = W$, then the problem~\eqref{eq:P} is self-bounded.
\end{lemma}

\begin{proposition}{\label{prop: Pbounded}}
	Let Assumption~\ref{assumption1}\rev{(a)-(c)} be satisfied. Then, problem~\eqref{eq:P} is bounded if and only if (P$_1(z)$) is bounded for each $z \in C^{+}$. 
\end{proposition}
\begin{proof}
	Note that the problem~\eqref{eq:P} is bounded if and only if it is self-bounded with $\mathcal{P}_{\infty} = \cl C$, see \cite[Lemma 2.2]{KovacovaRudloff2021}. Thus, the result follows from closedness of $C$ and Lemma~\ref{lemma_weights}.
\end{proof}
The following observation will be particularly useful in the polyhedral case.
\begin{remark} \label{rem:extdirC+}
	It follows from Proposition \ref{prop: Pbounded} and the convexity of $W$ from Lemma \ref{lemma_weights} that problem~\eqref{eq:P} is bounded if and only if (P$_1(z)$) is bounded for each extreme direction $z$ of $C^+$. 
\end{remark}

Now we discuss the second scalarization. For a point $v \in \mathbb{R}^q$ and a direction $d \in \mathbb R^q \setminus \{0\}$ the convex program 
\begin{align}
	\tag{P\textsubscript{2}$(v, d)$} 
	\label{P2(v)}
	\text{maximize } z \quad \text{ subject to } x\in\mathcal{X},\ z\in\real,\ \Gamma(x) - v - z d \leq_C 0
\end{align}
is the Pascoletti-Serafini scalarization (see~\cite{PasSer84}) with reference point $v$ and direction $-d$ as we have
\begin{align*}
	\sup\{ z\in\real \mid x \in \mathcal{X},\Gamma(x) - v - z d \leq_C 0 \} = - \inf \{ z\in\real \mid x\in\mathcal{X},\Gamma(x)-v+z d\leq_C 0 \}.
\end{align*}
Note that if $(x^*,z^*)$ is an optimal solution to \eqref{P2(v)}, then $x^*$ is a weak minimizer for \eqref{eq:P} (see e.g. \cite[Theorem 2.1.]{Eichfelder08}). The Lagrangian dual of problem~\eqref{P2(v)} is
\begin{align}
	\tag{D\textsubscript{2}$(v, d)$} 
	\label{D2(v)}
	\text{minimize } \sup_{x\in\mathcal{X}} \left\lbrace -\trans{w}\Gamma(x) \right\rbrace +\trans{w}v  \quad \text{ subject to } \trans{w}d=-1,\;w\geq_{C^+} 0 .
\end{align}

Now we provide some results concerning problems~\eqref{P2(v)} and~\eqref{D2(v)} that will be useful later.
\begin{proposition}\label{prop:strongduality}
	Let Assumption~\ref{assumption1}\rev{(a)-(c)} be satisfied and let one of the following hold
	\begin{enumerate}
		\item $v \in \Gamma(\mathcal{X}) + \Int C$ and $d \in \mathbb R^q \setminus \{0\}$, or
		\item $v \in \mathbb{R}^q$ and $d \in - \Int C$.
	\end{enumerate} 
	Then, strong duality holds for~\eqref{P2(v)} and~\eqref{D2(v)}. In particular, if~\eqref{P2(v)} is bounded, then \eqref{D2(v)} has an optimal solution. If \eqref{P2(v)} has an optimal solution, then their optimal values coincide. 
\end{proposition}
\begin{proof}
	Note that \eqref{P2(v)} is a convex programming problem. Hence, proving Slater's condition is sufficient. Let $\tilde{x} \in \Int \mathcal{X}$.
	We consider the two cases separately.
	\begin{enumerate}
		\item 
		Let $(x^*,z^*)$ be a feasible solution of \eqref{P2(v)}. Since $v \in \Gamma(\mathcal{X})+\Int C$, there exists $\bar{x}\in \mathcal{X}, \bar{c}\in \Int C$ such that $v = \Gamma(\bar{x})+\bar{c}$. Let $\lambda_1 \in (0,1)$. Since $\bar{c} \in \Int C$,
		there exists $0 < \lambda_2 \leq 1-\lambda_1$ such that $c_{\lambda}:=\lambda_1  \bar{c} + \lambda_2 (v-\Gamma(\tilde{x})) \in \Int C$. 
		Let $x_{\lambda}:= \lambda_1\bar{x}+\lambda_2 \tilde{x}+(1-\lambda_1-\lambda_2) x^*$. As $\Gamma$ is $C$-convex, we have
		\begin{align*}
			\Gamma(x_{\lambda}) &\leq_C \lambda_1 \Gamma(\bar{x})+\lambda_2 \Gamma(\tilde{x})+(1-\lambda_1-\lambda_2)\Gamma(x^*)\\
			&\leq_C \lambda_1(v- \bar{c})+\lambda_2 \Gamma(\tilde{x}) + (1-\lambda_1-\lambda_2)(v+z^*d)\\
			&= v + (1-\lambda_1-\lambda_2)z^*d -c_{\lambda}\\
			&<_C v + z_{\lambda}d,
		\end{align*}
		where $z_{\lambda}= (1-\lambda_1-\lambda_2)z^*$. Moreover, as $\lambda_2>0$, we also have $x_{\lambda}\in\Int\mathcal{X}$. Hence $(x_{\lambda},z_{\lambda})$ is a Slater point.
		
		\item  
		Note that since $C \subseteq \mathbb{R}^q$ is a solid convex cone, for $c \in \Int C$ and any $y \in \mathbb{R}^q$ there exists $t \in \mathbb{R}$ such that $y + t c \in \Int C$. Now consider $y=v-\Gamma (\tilde{x})$ for some $\tilde{x}\in  \mathcal{X} $ and $v\in \mathbb{R}^q $. This implies that there exists $\tilde{z}=-t \in \mathbb{R}$ such that $\Gamma (\tilde{x}) - v - \tilde{z}d \in - \Int C$, or equivalently $\Gamma (\tilde{x}) - v - \tilde{z}d <_{C} 0$, so $(\tilde{x}, \tilde{z})$ is a Slater point.  
	\end{enumerate}
	
\end{proof}

	
	\begin{proposition}
		\label{prop:1}
		Let Assumption~\ref{assumption1}\rev{(a)-(c)} be satisfied and let one of the following hold
		\begin{enumerate}
			\item $v \in \Gamma(\mathcal{X}) + \Int C$ and $d \in \mathbb R^q \setminus \{0\}$, or
			\item $v \in \mathbb{R}^q$ and $d \in - \Int C$.
		\end{enumerate} 
		Assume that $(x^{v,d},z^{v,d})$ and $w^{v,d}$ are optimal solutions to the problems \eqref{P2(v)} and \eqref{D2(v)}, respectively. Then $x^{v,d}$ is an optimal solution of (P$_1(w^{v,d})$) and $\alpha :=\trans{(w^{v,d})} \Gamma(x^{v,d})$ is its optimal objective value. Furthermore, $\set{y \in \real^q \mid \trans{(w^{v,d})} y = \alpha}$ is a supporting hyperplane of $\mathcal{P}$ at $\Gamma(x^{v,d})$ such that $\set{y \in \real^q \mid \trans{(w^{v,d})} y \geq \alpha} \supseteq \mathcal{P}$.
	\end{proposition}
	\begin{proof}
		\rev{It follows from strong duality between \eqref{P2(v)} and \eqref{D2(v)} together with \cite[Proposition 4.7]{LRU14}.} 
			%
	\end{proof}
	
	%
	
	\begin{proposition}
		\label{prop:P2unbounded}
		Let Assumption~\ref{assumption1}\rev{(a)-(c)} be satisfied and let $v \in \Gamma(\mathcal{X}) + \Int C$. 
		Problem~\eqref{P2(v)} is unbounded if and only if $d \in \mathcal{P}_{\infty}$. 
	\end{proposition}
	\begin{proof}
		If $d \in \mathcal{P}_{\infty}$, then $v + z d \in 
		\Int \mathcal{P}$ for all $z \geq 0$ as $v \in \Int \mathcal{P}$. Hence \eqref{P2(v)} is unbounded. On the other hand, if \eqref{P2(v)} is unbounded, then for all $z \geq 0$ there exists $x_z \in \mathcal{X}$ such that $\Gamma (x_z) \leq_C v + z d$. This implies $v + z d \in \Gamma (x_z) + C \subseteq \mathcal{P}$ for all $z \geq 0$. Since $\mathcal{P}$ is a closed convex set and $v \in \mathcal{P}$, this implies that $d$ is a recession direction of $\mathcal{P}$ (see Theorem 8.3 from \cite{Rockafellar72}).
	\end{proof}

	\begin{proposition}
		\label{prop:P2unbounded2}
		Let Assumption~\ref{assumption1}\rev{(a)-(c)} be satisfied, let $v \not\in \mathcal{P}$ and $d \in - \Int C$. 
		Then, problem~\eqref{P2(v)} is bounded.  
	\end{proposition}
	\begin{proof}
		First, note that by the argument in the proof of Proposition~\ref{prop:strongduality} the problem is feasible. Since $v \not\in \mathcal{P}$, the value $z = 0$ is not feasible for the problem. Assume that some $z > 0$ is feasible for the problem, i.e. for some $x \in \mathcal{X}$ it holds $ v + z d \in \Gamma(x) + C$. This would, however, yield a contradiction as $v = (v + zd) - zd \in \Gamma(x) + C + C \subseteq \mathcal{P}$. Therefore, zero is an upper bound on the optimal value of~\eqref{P2(v)}.
	\end{proof}

	\section{Determining an outer approximation}
	\label{sec_out_appr}
	The algorithm proposed in~\cite{LRU14} provides a finite $\varepsilon$--solution for a bounded convex vector optimization problem. It consists of two phases: First, an initial outer approximation $\mathcal{P}_0$ of $\mathcal{P}$ is found via~\cite[Equation 10]{LRU14}. Second, a sequence of outer approximations $\mathcal{P}_0 \supseteq \mathcal{P}_1 \supseteq \mathcal{P}_2\supseteq...\supseteq \mathcal{P}$ \rev{of the upper image} is computed until after finitely many steps an $\varepsilon$--approximation is obtained.
	
	Remark~\ref{remark_definition} suggests a strategy for dealing with unbounded problems: If a $\delta$--outer approximation of $\mathcal{P}_\infty$ can be found, then the algorithm of~\cite{LRU14} can be applied to the bounded problem~\eqref{eq:P'} to compute an $(\varepsilon,\delta)$--solution of~\eqref{eq:P}. 
	
	In this section, we formulate an algorithm which computes a $\delta$--outer approximation of $\mathcal{P}_\infty$ as well as an initial outer approximation $\mathcal{P}_0$ of $\mathcal{P}$ if the problem is feasible. In fact, the $\delta$--outer approximation of $\mathcal{P}_\infty$ will be the recession cone of the initial outer approximation $\mathcal{P}_0$ of $\mathcal{P}$. This means, in the proposed generalization of the initial phase an initial outer approximation $\mathcal{P}_0$ of $\mathcal{P}$ is computed whose recession cone is in $\delta$--distance to the recession cone of the upper image  $\mathcal{P}$ in the sense of ~\eqref{eq_def_1}. This is in contrast to the bounded case, where the initial outer approximation $\mathcal{P}_0$ of $\mathcal{P}$ has the same recession cone as the upper image $\mathcal{P}$.
	
	Thus, by replacing the original initial phase of~\cite[Algorithm 1]{LRU14} by this new procedure (Algorithm~\ref{alg_new} below), we are able to generalize~\cite[Algorithm 1 and Algorithm 2]{LRU14} to unbounded problems in Sections~\ref{sec_genBenson} and~\ref{sec_dual}.

	
	\rev{Recall that for computational purposes we are assuming that the ordering cone is polyhedral.}
	For a polyhedral ordering cone $C$ there exists a finite collection of generating directions, which will be denoted by $R$. Similarly, there exists a finite set of generating directions of the dual cone $C^+$, denoted by $R^*$ in the following. Note that if the generating directions $R$ are given, the set $R^*$ can be computed, and vice versa. So only one of them needs to be specified. Without loss of generality we assume that all elements of $R$ and $R^*$ are normalized such that $\norm{d} = 1$ for $d\in R$ and $\trans{c}w = 1$ for $w \in R^*$.
	
	Let us now describe the proposed algorithm, which is formalized as Algorithm~\ref{alg_new} below. Validity of the procedure is proven in Theorem~\ref{thm_alg_1} below.
	First, to verify the feasibility of~\eqref{eq:P} we solve the scalar problem \rev{$\text{minimize~} 0 \quad \text{subject to } x \in \mathcal{X}$.} 
	If the problem is feasible, we start with $\mathcal{P}_0 = \mathbb{R}^q$ as an outer approximation of $\mathcal{P}$. In the course of Algorithm~\ref{alg_new} this outer approximation will be refined until it satisfies the desired property that its recession cone is in $\delta$--distance to the recession cone of the upper image $\mathcal{P}$ in the sense of~\eqref{eq_def_1}.
	
	Next, we check whether problem~\eqref{eq:P} is bounded using Remark~\ref{rem:extdirC+}: We solve problem~\eqref{P1(w)} 
	for each $w \in R^*$. If all of these problems are bounded, then~\eqref{eq:P} is a bounded problem. Otherwise~\eqref{eq:P} is unbounded.  An optimal solution $x^*$ of a bounded problem~\eqref{P1(w)} 
	is a weak minimizer of~\eqref{eq:P} and as such is stored in a set $\bar{{\X}}$. The optimal solution $x^*$ and the weight $w$ also define a supporting hyperplane of $\mathcal{P}$, which is used to update the outer approximation $\mathcal{P}_0$. 
	
	If~\eqref{eq:P} is found to be bounded, the algorithm can terminate, as $C$ (or its generating directions $R$) provide an outer approximation of $\mathcal{P}_{\infty}$ and thus the computation of the initial outer approximation $\mathcal{P}_0$ of $\mathcal{P}$ coincides with the initial outer approximation \rev{of the upper image} in~\cite[Equation 10]{LRU14}. Otherwise, an outer approximation of $\mathcal{P}_{\infty}$ needs to be iteratively computed through refining the outer approximation $\mathcal{P}_0$ of $\mathcal{P}$ step by step. We keep track of all inner directions (elements of $\mathcal{P}_{\infty}$) found in a set $\mathcal{Y}_{In}$, initialized with the generators $R$ of the ordering cone $C$.
	
	In each iteration, the recession cone $(\mathcal{P}_0)_{\infty}$ of the current outer approximation $\mathcal{P}_0$ is computed. This provides an outer approximation of the recession cone of the upper image in the sense of $\mathcal{P}_{\infty} \subseteq (\mathcal{P}_0)_{\infty}$. If the distance (within the unit ball as given in~\eqref{eq_def_1}) is at most $\delta$, the algorithm can terminate. Otherwise, we need to improve the approximation $\mathcal{P}_{0}$. To verify if condition~\eqref{eq_def_1} holds, we specifically compute distances between the elements of the set $\mathcal{Y}_{Out} := \vertice \left( (\mathcal{P}_0)_{\infty} \cap B_1 (0) \right) \setminus \{0\}$ and the known inner directions $\mathcal{Y}_{In}$. If~\eqref{eq_def_1} is not satisfied, we consider a convex combination of an outer and an inner direction, which either updates the set of known inner directions or provides a supporting hyperplane of $\mathcal{P}$ to update $\mathcal{P}_0$.
	
	A precise formulation of this algorithm is provided below. It contains a matrix $T$ and a set $\bar{\Tset}$, both of which will be discussed in Section \ref{sec_dual} as part of the dual algorithm. For a primal version of the algorithm these can be ignored and thus are colored in gray.

	\rev{\begin{remark}\label{rem:lineality} The following modifications can be made to Algorithm~\ref{alg_new} without affecting the subsequent results.
		\begin{enumerate}
			\item As an optional step, it is possible to check for lineality directions before starting the iterations of Algorithm~\ref{alg_new}. The motivation is to fasten the process of finding the lineality directions of the upper image, if there are any, or to obtain a better initial approximation, otherwise. This can be done by applying the following procedure, right before executing line 16 of the algorithm.\\
			For all $d \in -R$ solve problem~\eqref{P2(v)} and its dual~\eqref{D2(v)} .
			\begin{enumerate}
				\item If~\eqref{P2(v)} is unbounded, then set $\mathcal{Y}_{In}\gets \mathcal{Y}_{In}\cup\set{d}$.
				\item If $(x^{v,d},z^{v,d})$, $w^{v,d}$ are optimal solutions to \eqref{P2(v)}, \eqref{D2(v)}, respectively, then set $\mathcal{\bar{\X}}\gets \mathcal{\bar{\X}}\cup\set{x^{v,d}}$, \g{$\bar{\Tset}\gets\bar{\Tset}\cup\set{\trans{T}\frac{w^{v,d}}{\trans{c} w^{v,d}} }$,} and $$\mathcal{P}_0\gets\mathcal{P}_0\cap\set{y\in\real^q\mid \trans{(w^{v,d})} (y-\Gamma(x^{v,d}))\geq0}.$$
			\end{enumerate}
			\item In each iteration (in line 18) one chooses an outer direction $d$ with which the iteration proceeds. Currently, this direction is chosen arbitrarily. Instead, one could formulate selection rules for the direction. Another alternative would be to check all directions $d \in \mathcal{Y}_{Out} \setminus \left( \mathcal{Y}_{In} \cup \hat{C} \right)$ before updating the outer approximation $\mathcal{P}_0$. We leave these considerations for future work.
			\item In line 22 of Algorithm~\ref{alg_new} one could alternatively choose the direction $\tilde{d} := \frac{\lambda d+(1-\lambda)\tilde{r}}{\norm{\lambda d+(1-\lambda)\tilde{r}}}$ for any $\lambda \in (0,1)$.
		\end{enumerate}
\end{remark}}

{\scriptsize 
			\begin{center}
	\begin{minipage}{\linewidth}
		\begin{algorithm}[H]  
		\SetAlgoLined
		\KwIn{ CVOP satisfying \rev{Assumption~\ref{assumption1}}; generating directions $R$ of $C$ with $\norm{d} = 1$ for $d \in R$; generating directions $R^*$ of $C^+$ with $\trans{c} w = 1$ for $w \in R^*$; tolerance $\delta > 0$. }
Set $\mathcal{Y}_{In}:=R$, $\hat{C}:=\emptyset$, $\mathcal{P}_0:=\real^q$, $\mathcal{\bar{\X}}:=\emptyset$, $bounded := $TRUE, \g{$\bar{\Tset}=\emptyset$}. \;

		\tcp{Check Feasibility:}
		
		\eIf{$\min_{x \in \mathcal{X}} 0 $ is infeasible} {
		Set $\mathcal{Y}_{Out} := \emptyset$\;}
		{Use optimal (feasible) solution $x_0$ to compute $v := \Gamma (x_0) + c$ \;}
		\tcp{Check Boundedness:}
		\For{$w \in R^*$}{  
		\eIf{$P_1(w)$ is unbounded}{ 
$bounded := $FALSE}{ 		
Use optimal solution $x^*$ of $P_1(w)$ to update $\mathcal{\bar{\X}}\gets\mathcal{\bar{\X}}\cup\set{x^*}$, \g{$\bar{\Tset}\gets\bar{\Tset}\cup\set{\trans{T} w}$,} and  $\mathcal{P}_0\gets\mathcal{P}_0\cap\set{y\in\real^q \mid (\trans{w} (y-\Gamma(x^*))\geq0}$. }
		\eIf{$bounded$}{
Set $\mathcal{Y}_{Out}:=R$}{
Compute $\mathcal{Y}_{Out} = \vertice \left( (\mathcal{P}_0)_{\infty} \cap B_1 (0) \right) \setminus \{0\}$.
	\tcp{Iteration:}
		\While{$\mathcal{Y}_{Out} \setminus \left( \mathcal{Y}_{In} \cup \hat{C} \right) \neq \emptyset$}{
Take $d \in \mathcal{Y}_{Out} \setminus \left( \mathcal{Y}_{In} \cup \hat{C} \right)$  and find $\tilde{r} := \text{argmin} \left\lbrace \norm{ d-r} \; \big| \; r\in \mathcal{Y}_{In} \right\rbrace$.\;

		\eIf{$\norm{ d-\tilde{r}} \leq \delta$}{
Update $\hat{C}\gets\hat{C} \cup \set{d}$.}{
\rev{Set $\tilde{d} := \frac{d+\tilde{r}}{\norm{{d+\tilde{r}}}}$.} 
		\eIf{$P_2 (v, \tilde{d})$ is unbounded }{
Update $\mathcal{Y}_{In} \gets \mathcal{Y}_{In}\cup\set{\tilde{d}}$. }{
Use optimal solution $(x^{v,\tilde{d}},z^{v,\tilde{d}})$ and $w^{v,\tilde{d}}$ of $P_2(v,\tilde{d})$ and $D_2(v,\tilde{d})$ to update $\mathcal{\bar{\X}}\gets \mathcal{\bar{\X}}\cup\set{x^{v,\tilde{d}}}$, \g{$\bar{\Tset}\gets\bar{\Tset}\cup\set{\trans{T}\frac{w^{v,\tilde{d}}}{\trans{c} w^{v,\tilde{d}}} }$,} and \\
		$\mathcal{P}_0 \gets \mathcal{P}_0 \cap \set{y\in\real^q \mid \trans{(w^{v,\tilde{d}})} (y-\Gamma(x^{v,\tilde{d}}))\geq0}$.
		
Recompute $\mathcal{Y}_{Out} \gets \vertice \left( (\mathcal{P}_0)_{\infty} \cap B_1 (0) \right) \setminus \{0\}$.}
}
}
}
}
\KwOut{	Initial outer approximation $\mathcal{P}_0$ of the upper image of CVOP; set of weak minimizers $\bar{\mathcal{X}}$ and \g{set of maximizers $\bar{\mathcal{T}}$}; sets $\mathcal{Y}_{In}$, $\mathcal{Y}_{Out}$ satisfying $\cone \mathcal{Y}_{In} \subseteq \mathcal{P}_{\infty} \subseteq \cone \mathcal{Y}_{Out}$. }
	\caption{\label{alg_new}\rev{Determining a finite $\delta$--outer approximation of $\mathcal{P}_{\infty}$ and  an initial outer approximation $\mathcal{P}_0$ of the upper image $\mathcal{P}$}} 
\end{algorithm}
\end{minipage}
\end{center}
}
	
	The following observations are immediate.
	\begin{lemma} \label{lemma_alg_prop} 
		Let Assumption~\ref{assumption1} be satisfied.
		\begin{enumerate}
			\item The problem~\eqref{eq:P} is infeasible if and only if Algorithm~\ref{alg_new} \rev{never executes lines 5-31.} 
			\item The problem~\eqref{eq:P} is bounded if and only if Algorithm~\ref{alg_new} \rev{never executes lines 16-30.} 
			\item The set $\mathcal{P}_0$ is (in each step of Algorithm~\ref{alg_new}) a convex polyhedron that is an outer approximation of the upper image $\mathcal{P}$, i.e. it satisfies $\mathcal{P} \subseteq \mathcal{P}_0$.
			\item The set $\mathcal{Y}_{In}$ contains (in each step of Algorithm~\ref{alg_new}) only recession directions of the upper image, i.e. it satisfies $\mathcal{Y}_{In} \subseteq \mathcal{P}_{\infty}$.
		\end{enumerate}
	\end{lemma}
	\begin{proof}
		\begin{enumerate}
			\item Trivial, as the problem $\min_{x \in \mathcal{X}} 0$ shares the feasible set of~\eqref{eq:P}.
			
			\item This follows from Proposition~\ref{prop: Pbounded} and Lemma~\ref{lemma_weights}.
			
			\item The set $\mathcal{P}_0$ is initialized as the whole space, $\mathcal{P}_0:=\real^q$. Whenever the set $\mathcal{P}_0$ is updated, it is through an intersection with a halfspace generated by a supporting hyperplane of $\mathcal{P}$, see Proposition~\ref{prop:1}. Since $\mathcal{P}_0$ is an intersection of finitely many halfspaces, it is a convex polyhedron.
			
			\item The set $\mathcal{Y}_{In}$ is initialized with generating directions of the ordering cone, which are contained in $\mathcal{P}_{\infty}$. A direction $d$ is added to the set $\mathcal{Y}_{In}$ only when problem~\eqref{P2(v)} is unbounded, according to Proposition~\ref{prop:P2unbounded} such directions are recession directions of the upper image.
		\end{enumerate}
	\end{proof}
	
	The following theorem shows that Algorithm~\ref{alg_new} provides a finite $\delta$--outer approximation of the recession cone $\mathcal{P}_{\infty}$, as well as that the outer approximation $\cone \mathcal{Y}_{Out}$ of the recession cone leads to a bounded problem~\eqref{eq:P'}. 
	\begin{theorem}
		\label{thm_alg_1}
		Assume that Algorithm~\ref{alg_new} terminated after finitely many steps for a feasible problem~\eqref{eq:P} satisfying Assumption~\ref{assumption1}.
		Then, the set $\mathcal{Y}_{Out}$ outputted by Algorithm~\ref{alg_new} is a finite $\delta$--outer approximation of the recession cone $\mathcal{P}_{\infty}$ and the problem~\eqref{eq:P'} is bounded.  The set $\mathcal{Y}_{In}$ outputted by Algorithm~\ref{alg_new} is a finite $\delta$--inner approximation of the recession cone $\mathcal{P}_{\infty}$.
	\end{theorem}
	\begin{proof}
		When the algorithm terminates after finitely many steps, $\mathcal{P}_0$ is a convex polyhedron and $\mathcal{Y}_{Out} = \vertice \left( (\mathcal{P}_0)_{\infty} \cap B_1 (0) \right) \setminus \{0\}$ is a finite set. The set $\mathcal{Y}_{In}$ is also finite. From Lemma~\ref{lemma_alg_prop} we know that it holds $\mathcal{P} \subseteq \mathcal{P}_0$ and
		\begin{align}
			\label{eq_thm_alg_0}
			\cone \mathcal{Y}_{In} \subseteq \mathcal{P}_{\infty} \subseteq (\mathcal{P}_0)_{\infty} = \cone \mathcal{Y}_{Out}.
		\end{align}
		Upon termination, the set $\mathcal{Y}_{Out}$ satisfies $\mathcal{Y}_{Out} \subseteq \mathcal{Y}_{In} \cup \hat{C}$, where $\hat{C}$ contains only vectors with distance at most $\delta$ to the set $\mathcal{Y}_{In}$. Thanks to this we have
		\begin{align}
			\label{eq_thm_alg_1}
			\forall d \in \mathcal{Y}_{Out} \;\; \exists \tilde{r} \in \mathcal{Y}_{In} \; : \; \norm{ d-\tilde{r}} \leq \delta.
		\end{align}
		We also know that $\mathcal{Y}_{Out} = \vertice (\cone \mathcal{Y}_{Out} \cap B_1(0) ) \setminus \{0\}$, since $\cone \mathcal{Y}_{Out} = (\mathcal{P}_0)_{\infty}$, and that all elements of $\mathcal{Y}_{Out}$ and $\mathcal{Y}_{In}$ are normalized. Now we prove that this implies
		\begin{align}
			\label{eq_thm_alg_2}
			\forall d \in \left( \cone \mathcal{Y}_{Out} \cap B_1 (0) \right), \; \exists \tilde{r} \in \left( \cone \mathcal{Y}_{In} \cap B_1 (0) \right) \; : \; \norm{ d-\tilde{r}} \leq \delta.
		\end{align}
		Take a nonzero $d \in \left( \cone \mathcal{Y}_{Out} \cap B_1 (0) \right)$. Since $\left( \cone \mathcal{Y}_{Out} \cap B_1 (0) \right)$ is the convex hull of $\mathcal{Y}_{Out}$ and the vector $0$, there exist $d^{(1)}, \dots, d^{(k)} \in \mathcal{Y}_{Out}$ and $\alpha^{(1)}, \dots, \alpha^{(k)} > 0$ with $\sum_{i=1}^k \alpha^{(i)} \leq 1$ such that $d = \sum_{i=1}^k \alpha^{(i)} d^{(i)}$. According to~\eqref{eq_thm_alg_1}, for each $d^{(i)}$ there exists a corresponding $\tilde{r}^{(i)} \in \mathcal{Y}_{In}$ with $\norm{ d^{(i)} - \tilde{r}^{(i)} } \leq \delta$. Then, $\tilde{r} = \sum_{i=1}^k \alpha^{(i)} \tilde{r}^{(i)} \in \cone \mathcal{Y}_{In}$. Convexity of the norm shows both that $\tilde{r} = \sum_{i=1}^k \alpha^{(i)} \tilde{r}^{(i)} \in B_1(0)$ and that $\norm{ d-\tilde{r}} \leq \delta$, which proves~\eqref{eq_thm_alg_2}. Together with~\eqref{eq_thm_alg_0} this shows that 
		\begin{align*}
			\text{d}_H \left( \cone \mathcal{Y}_{In} \cap B_1(0), \cone\mathcal{Y}_{Out}\cap B_1(0) \right) \leq \delta,
		\end{align*}
		which by~\eqref{eq_thm_alg_0} implies that $\mathcal{Y}_{Out}$ ($\mathcal{Y}_{In}$) is a finite $\delta$--outer (inner) approximation of $\mathcal{P}_{\infty}$.
		Finally, to prove that problem~\eqref{eq:P'} is bounded, we prove that the set $\mathcal{P}_0$ is self-bounded, i.e. it satisfies $\mathcal{P}_0 \subseteq \{p_0\} + (\mathcal{P}_0)_{\infty}$ for some $p_0 \in \mathbb{R}^q$. 
		The set $\mathcal{P}_0$ is a convex polyhedron, therefore, it has an H-representation
		$$\mathcal{P}_0 = \bigcap_{i=1}^r \left\lbrace y \in \mathbb R^q \mid \trans{(w^i)} y \geq \gamma_i\right\rbrace$$
		for some $r \in \mathbb{N}, w^1,...,w^r \in \mathbb R^q \setminus \{0\}$ and $\gamma_1,...,\gamma_r \in \mathbb{R}$. Its recession cone is
		$$(\mathcal{P}_0)_{\infty} = \bigcap_{i=1}^r \left\lbrace y \in \mathbb R^q \mid \trans{(w^i)} y \geq 0 \right\rbrace.$$ The choice of $p_0 := \min_{i=1, \dots, r} \left\lbrace \frac{\gamma_i}{\trans{(w^i)} c}
		\right\rbrace c$ for $c \in \Int C \subseteq \Int (\mathcal{P}_0)_{\infty}$ gives the desired result.
	\end{proof}
	
	Since each direction $d$ considered throughout Algorithm~\ref{alg_new} is normalized (before being added to a set of outer directions or a set of inner directions), the angles between the recession cone $\mathcal{P}_{\infty}$ and the outer approximation $\cone \mathcal{Y}_{Out}$ \rev{of $\mathcal{P}_{\infty}$} are bounded as described in Remark~\ref{remark_angles}.
	
		\begin{remark} \label{remark_finiteness}
			Independently of this work, an algorithm for approximating recession directions of so-called spectahedral shadows was developed in~\cite{DorLoh2022}. 
			The proof of termination in a finite number of steps from~\cite[Theorem 4.4]{DorLoh2022} can be applied also to Algorithm~\ref{alg_new} if we modify the iteration by replacing lines 17-30 with lines 17-30 displayed in Algorithm~\ref{alg_modified}.
			{\scriptsize 
			\begin{center}
			\begin{minipage}{\linewidth}
			\begin{algorithm}[H]  
				\SetAlgoLined
				\setcounter{AlgoLine}{16}
				\caption{{\label{alg_modified}}\rev{Modified iteration of Algorithm~\ref{alg_new}.}}
					\Repeat{$H == \real^q$}{
						Set $H:= \real^q$.\\
					\For{$d \in \mathcal{Y}_{Out}$}{
						\For{$k = 1 : \Big\lceil \log_2 \frac{\norm{d - c}}{\delta} \Big\rceil$}{
								Set $d_k := \frac{2^k-1}{2^k}d + \frac{1}{2^k}c$. \\
								\eIf{$P_2 (v, d_k)$ is unbounded }{
									Update $\mathcal{Y}_{In} \gets \mathcal{Y}_{In}\cup\set{\frac{d_k}{\norm{d_k}}}$. }{
									Use optimal solutions $(x^{v,d_k},z^{v,d_k})$ and $w^{v,d_k}$ of $P_2(v,d_k)$ and $D_2(v,d_k)$ to update $\mathcal{\bar{\X}}\gets \mathcal{\bar{\X}}\cup\set{x^{v,d_k}}$, \g{$\bar{\Tset}\gets\bar{\Tset}\cup\set{\trans{T}\frac{w^{v,d_k}}{\trans{c} w^{v,d_k}} }$,} and 
									$H \gets H\cap\set{y\in\real^q \mid \trans{(w^{v,d_k})} (y-\Gamma(x^{v,d_k}))\geq0}$.
									}
									}
									}
Update $\mathcal{P}_0 \gets \mathcal{P}_0 \cap H$ and $\mathcal{Y}_{Out} \gets \vertice \left( (\mathcal{P}_0)_{\infty} \cap B_1 (0) \right) \setminus \{0\}$. 
}
			\end{algorithm}
\end{minipage}
\end{center} }
		\end{remark}

	\section{Algorithms for solving unbounded CVOPs}
	\label{sec_algs}
	A primal and a dual algorithm for solving convex vector optimization problems~\eqref{eq:P} that are allowed to be unbounded are presented in the following subsections.
	\subsection{Primal algorithm}
	\label{sec_genBenson}
	A finite $\delta$--outer approximation of $\mathcal{P}_\infty$ which defines a bounded problem \eqref{eq:P'} is computed by Algorithm \ref{alg_new}. 
	It also provides an initial outer approximation $\mathcal{P}_0$ of $\mathcal{P}$ which will replace the initialization phase~\cite[Equation 10]{LRU14} of~\cite[Algorithm 1]{LRU14} when dealing with unbounded problems.
	To compute a solution for the general convex problem \eqref{eq:P} the iterative part of~\cite[Algorithm 1]{LRU14} can be used for~\eqref{eq:P'}, where $\mathcal{Y}$ is set to $\mathcal{Y}_{Out}$, see Remark~\ref{remark_definition}. This generalized solution algorithm is summarized in this section.
	Similar to Algorithm~\ref{alg_new}, it contains a matrix $T$ and a set $\bar{\Tset}$, both of which will be discussed in Section \ref{sec_dual} as part of the dual algorithm. For a primal version of the algorithm these can be ignored and thus are colored in gray.

	{\scriptsize 
		\begin{center}
			\begin{minipage}{\linewidth}
				\begin{algorithm}[H]  
					\SetAlgoLined
			\caption{\label{alg_2}\rev{Primal Algorithm for solving convex vector optimization problems }}
\KwIn{Problem~\eqref{eq:P} satisfying \rev{Assumption~\ref{assumption1}}, tolerances $\varepsilon, \delta > 0$. } 
Apply Algorithm~\ref{alg_new} with tolerance $\delta$ to problem~\eqref{eq:P}: Compute set $\mathcal{Y}_{Out}$; $H$-representation of the initial outer approximation $\mathcal{P}_0$ \rev{of $\mathcal{P}$}; set of weak minimizers $\bar{\X}$; \g{set of maximizers $\bar{\Tset}$.} \\
Set $k=0$, $M = \emptyset$, set $d = -c \in -\Int C$.\\
					\While {$M \neq \mathbb R^q$}{
\rev{Set $M:=\real^q$.}\\
Compute the set $\mathcal{P}^V$ of vertices of $\mathcal{P}_k$.\\
					\For{$v\in\mathcal{P}^V$}{
Compute optimal solutions $(x^{v,d},z^{v,d})$ and $w^{v,d}$ to Problems \eqref{P2(v)} and \eqref{D2(v)}, respectively.\\
$\bar{\X}\gets\bar{\X}\cup\set{x^{v,d}}$\g{, $\bar{\Tset}\gets\bar{\Tset}\cup\set{\trans{T} w^{v,d}}$}. \\
					\If{ $z^v> \varepsilon$} { 
					$M\gets M\cap\set{y\in\real^q \mid \trans{(w^{v,d})} (y-\Gamma(x^{v,d}))\geq0}$.}  
}
					\If {$M\neq\real^q$}{ 
$\mathcal{P}_{k+1}\gets\mathcal{P}_k\cap M$ and $k\gets k+1$.
}
}
\KwOut{Set of weak minimizers $\bar{\X}$, set of directions $\mathcal{Y}_{Out}$, \g{set of maximizers $\bar{\mathcal{T}}$}. }
	\end{algorithm}
\end{minipage}
\end{center} }

	The next theorem states that Algorithm \ref{alg_2} outputs a finite weak $(\varepsilon, \delta)$--solution of~\eqref{eq:P} if it terminates. \rev{In Remark~\ref{remark_near_opt} this assumption will be discussed and modifications of the algorithm are given that guarantee finite termination.}
	\begin{theorem}
		\label{thm_alg_2}
		Assume that Algorithm~\ref{alg_2} terminated for a feasible problem~\eqref{eq:P} satisfying Assumption~\ref{assumption1}.
		Then, the pair $(\bar{\mathcal{X}}, \mathcal{Y}_{Out})$ outputted by Algorithm~\ref{alg_2} is a finite weak $(\varepsilon, \delta)$--solution of~\eqref{eq:P}. 
	\end{theorem}
	\begin{proof}
		According to Theorem~\ref{thm_alg_1}, the set $\mathcal{Y}_{Out}$ is a finite $\delta$--outer approximation of the recession cone $\mathcal{P}_{\infty}$ and problem~\eqref{eq:P'} is bounded. Since problem~\eqref{eq:P} satisfies Assumption~\ref{assumption1} and the ordering cone $\cone \mathcal{Y}_{Out}$ is polyhedral, problem~\eqref{eq:P'} satisfies~\cite[Assumption 4.1]{LRU14}, except for the feasible region $\mathcal{X}$ to be compact, which is discussed in~\cite[Remark~3,~Section~4.3]{LRU14}. Then, according to~\cite[Theorem 4.9 and Section~4.3]{LRU14} if \cite[Algorithm 1]{LRU14} terminates, it delivers a finite weak $\varepsilon$--solution $\bar{\X}$ of~\eqref{eq:P'} in the sense of~\cite[Definition 3.3]{LRU14}, i.e. it satisfies
		\begin{align*}
			\cl \left( \Gamma (\mathcal{X}) + \cone \mathcal{Y}_{Out} \right) \subseteq \conv \Gamma (\bar{\X}) + \cone \mathcal{Y}_{Out} - \varepsilon \{c\}.
		\end{align*}
		\rev{Lines 2-16} of Algorithm~\ref{alg_2} coincide with the steps of~\cite[Algorithm 1]{LRU14}, the only difference is in the initial outer approximation $\mathcal{P}_0$ \rev{of $\mathcal{P}$} computed in \rev{line}~1. 
		Since $\mathcal{P} \subseteq \cl \left( \Gamma (\mathcal{X}) + \cone \mathcal{Y}_{Out} \right)$ this implies
		\begin{align*}
			\mathcal{P} \subseteq \conv \Gamma (\bar{\X}) + \cone \mathcal{Y}_{Out} - \varepsilon \{c\}.
		\end{align*}
		Finally, the points $\Gamma (\bar{\X})$ are \rev{weakly $\cone \mathcal{Y}_{Out}$-minimal} in $\cl \left( \Gamma (\mathcal{X}) + \cone \mathcal{Y}_{Out} \right)$. As $C \subseteq \cone \mathcal{Y}_{Out}$ this implies that the points $\Gamma (\bar{\X})$ are \rev{weakly $C$-minimal} in $\mathcal{P}.$
		This jointly shows that $(\bar{\X}, \mathcal{Y}_{Out})$ is a finite weak $(\varepsilon, \delta)$--solution of~\eqref{eq:P}.
	\end{proof}

	\begin{remark}
		Let's consider a bounded problem~\eqref{eq:P} for which Algorithm~\ref{alg_2} was applied, yielding a finite weak $(\varepsilon, \delta)$-solution as defined in Definition~\ref{solutionconcept}. For a bounded problem~\eqref{eq:P} the solution concept of~\cite[Definition 3.3]{LRU14} (see~\eqref{sol-concept_bd} above) is also applicable. So, is there a connection between the output of Algorithm~\ref{alg_2} and the more restrictive definition of a solution for a bounded problem? Yes, the set of feasible points $\bar{\X}$ outputted by Algorithm~\ref{alg_2} is a finite weak $\varepsilon$-solution of~\eqref{eq:P} in the sense of~\cite[Definition 3.3]{LRU14} (respectively~\eqref{sol-concept_bd}). This follows from the fact that for a bounded problem~\eqref{eq:P} Algorithm~\ref{alg_new} outputs generators of the ordering cone $C$ as directions $\mathcal{Y}_{Out}$, see Lemma~\ref{lemma_alg_prop}. 
	\end{remark}
	
	\begin{remark}
		\label{remark_near_opt}
		In Theorems~\ref{thm_alg_1} and~\ref{thm_alg_2} we prove that Algorithm~\ref{alg_new}, respectively Algorithm~\ref{alg_2}, yields the desired output under the assumption that the algorithm terminated successfully. Let us now discuss this assumption. 
		It is equivalent to saying that (a) for all involved scalarizations that are bounded a solution exists and (b) the algorithm terminates in finitely many steps. 
		\rev{
			Let us start with part (b). We address the finite termination of the algorithm for approximating recession directions (Algorithm~\ref{alg_new}) in Remark~\ref{remark_finiteness}. For Algorithm~\ref{alg_2} we would additionally need to guarantee finite termination for the second phase of the algorithm that is based on~\cite[Algorithm 1]{LRU14}. Recently,~\cite{AraUlusUmer2021} proposed a direction-free modification of~\cite[Algorithm 1]{LRU14} for which termination in a finite number of steps can be proven. Combining the modification proposed in Remark~\ref{remark_finiteness} with a second phase based on~\cite[Algorithm 2]{AraUlusUmer2021} would yield an algorithm for which a termination after a finite number of steps can be proven. Note, however, that~\cite[Algorithm 2]{AraUlusUmer2021} is based on a slightly different, direction-free, solution concept.
		}
		
		Let us now consider part (a) of this assumption. Throughout the two algorithms the three scalarizations~\eqref{P1(w)}, \eqref{P2(v)} and \eqref{D2(v)} are considered, they are all feasible and the considered scalarizations are solved whenever they are bounded: During the iterations of Algorithm~\ref{alg_2} all considered problems \eqref{P2(v)} are bounded according to Proposition~\ref{prop:P2unbounded2}, and by Proposition~\ref{prop:strongduality}(2), for all considered problems \eqref{D2(v)} a solution exists. Algorithm~\ref{alg_new} considers the scalarizations~\eqref{P1(w)}, \eqref{P2(v)} and always distinguishes between the scalarization being bounded or unbounded. If it is bounded, a solution is sought. But unlike~\cite{LRU14}, we do not assume a compact feasible set, so we cannot guarantee the existence of an optimal solution of the considered bounded scalarizations in both algorithms. However, whenever the scalarized problem is bounded, a near-optimal solution exists for any desired level of accuracy. 
		
		In practice, it is even less important, if a solution or just a near-optimal solution is computed as (i) the level of accuracy of the solvers are typically much smaller than the values of $\delta$ and $\varepsilon$ considered in Algorithm~\ref{alg_new} and~\ref{alg_2} and (ii) within an implementation, a bounded scalar problem will be solved up to the given level of accuracy (i.e. near-optimally solved) regardless of whether an optimal solution exists or not. In the following, we will discuss the impact of the level of accuracy of these near-optimal solutions on Algorithm~\ref{alg_new} and~\ref{alg_2}.

		Denote by $\epsilon \ll \min \{ \delta, \varepsilon \}$ the accuracy for solving scalar problems. Let us start with the (bounded) weighted sum scalarization. Assume that $\tilde{x}$ is feasible (i.e.~$\tilde{x} \in \mathcal{X}$) and $\epsilon$-optimal (i.e.~$\vert \trans{w} \Gamma (\tilde{x}) - \inf \{ \trans{w} \Gamma (x) \mid x \in \mathcal{X} \} \vert \leq \epsilon $) for problem~\eqref{P1(w)}. \rev{Then, it holds}
		\begin{align}\label{eq:shifted1}
			\Gamma (\mathcal{X}) \subseteq \{ y \in \real^q  \mid \trans{w} y \geq \trans{w} \Gamma (\tilde{x}) - \epsilon \}.
		\end{align}
		
		Now consider the (bounded) Pascoletti-Serafini scalarization. Assume that $(\tilde{x}, \tilde{z})$ and $\tilde{w}$ are feasible (for~\eqref{P2(v)} and~\eqref{D2(v)}, respectively) and $\epsilon$-optimal (i.e. $\vert \sup_{x \in \mathcal{X}} \{ - \tilde{w}^\T \Gamma (x) \} + \tilde{w}^\T v - \tilde{z} \vert \leq \epsilon$) for the pair of dual problems \eqref{P2(v)} and \eqref{D2(v)}. \rev{Then, using the weak duality between~\eqref{P2(v)} and~\eqref{D2(v)},  the fact that $\tilde{x} \in \mathcal{X}$ and the $\epsilon$-optimality of $(\tilde{x}, \tilde{z})$ and $\tilde{w}$, we obtain
			\begin{align*}
				\tilde{z} \leq - \tilde{w}^\T \Gamma (\tilde{x}) + \tilde{w}^\T v \leq \sup\limits_{x \in \mathcal{X}} \{ - \tilde{w}^\T \Gamma (x) \} + \tilde{w}^\T v \leq \tilde{z} + \epsilon. 
			\end{align*} 
			Rearranging the terms, one can show that
			\begin{align*}
				\tilde{w}^\T \Gamma (x) \geq \tilde{w}^\T v - \tilde{z} - \epsilon \geq \tilde{w}^\T \Gamma (\tilde{x}) - \epsilon
			\end{align*}
			holds for any $x \in \mathcal{X}$. This implies }		
		\begin{align}\label{eq:shifted2}
			\Gamma (\mathcal{X}) \subseteq \{ y \in \real^q \mid \tilde{w}^\T y \geq \tilde{w}^\T \Gamma (\tilde{x}) - \epsilon  \}.
		\end{align}
		
		Within the algorithms, the quantity $-\epsilon$ does not appear in the intercept of the halfspaces. This means that the \rev{halfspaces that are used to construct the outer approximation are slightly shifted versions of the ones given by \eqref{eq:shifted1} and \eqref{eq:shifted2}.} They are, however, not tilted, which means that the recession directions searched within Algorithm~\ref{alg_new} are not affected. We could incorporate the quantity $-\epsilon$ into the implementation of Algorithm~\ref{alg_2}. But this quantity is in practice significantly smaller than the target tolerance $\varepsilon$. 
		Thus, for practical purposes, the assumptions in Theorems~\ref{thm_alg_1} and~\ref{thm_alg_2} boil down to the assumption that the algorithm terminates in finitely many steps, \rev{the usual assumptions in the literature. Above we also referred to modifications that allow for proving finiteness.} 
		The same holds true for the dual algorithm considered in the next subsection under an additional condition, see Remark~\ref{rem:condition_dual}. 
	\end{remark}

		\subsection{Dual algorithm}
		\label{sec_dual}
		
		The theory of geometric duality for CVOPs was developed by Heyde in~\cite{Hey13}.
		We shortly recall some basic facts here, but we refer the reader to \cite{Hey13, LRU14} for more details.
		A geometric dual algorithm to solve bounded CVOPs was proposed in \cite{LRU14}. In fact, both algorithms from \cite{LRU14} solve both the primal and the geometric dual problem simultaneously. We will show that the same holds true for the primal and dual algorithm for unbounded CVOPs proposed here.
		
		In this section, we show that the geometric dual algorithm~\cite[Algorithm 2]{LRU14} can also be used in order to solve unbounded problems when the following small modifications are taken into account: (i) we will propose a generalized solution concept for the dual problem~\eqref{(D)} defined below, (ii) similarly to the primal algorithm, the initialization phase of~\cite[Algorithm 2]{LRU14} is replaced by a new initialization based on Algorithm \ref{alg_new}, (iii) the unbounded problem~\eqref{eq:P} is transformed into a bounded problem~\eqref{eq:P'}, see Section~\ref{sec_transfo} and (iv)~\cite[Algorithm 2]{LRU14} is used to solve the geometric dual~\eqref{(D')} of the bounded problem~\eqref{eq:P'}. 
		Unlike the primal variant, the dual algorithm works under an additional condition discussed in Remark~\ref{rem:condition_dual} and \rev{line}~2 of Algorithm~\ref{alg_dual}. This condition is always satisfied for self-bounded problems.
		
		Before providing the definition of the geometric dual problem, let us introduce some notation. Recall that $c\in \Int C$ is fixed. Further, we fix a nonsingular matrix $T\in \real^{q\times q}$ such that the last column of $T$ is the vector $c$.
		This is the matrix appearing in (the gray parts of) Algorithms \ref{alg_new} and \ref{alg_2}. The geometric dual problem of problem~\eqref{eq:P} is given by
		\begin{align*}
			\label{(D)}
			\text{maximize~} D^*(t) \quad \text{with respect to~} \leq_K \quad \text{subject to~}  w(t) \in C^+ \tag{D},
		\end{align*}
		where the ordering cone is $K:= \{\trans{(0,0,\ldots,0,k)} \in \real^q \mid k \geq 0\}= \real_+e^q$, 
		the function $w: \real^q \to \real^q$ is defined as 
		\begin{equation}\label{eq:w(t)}
			w(t):= \trans{\left(\left( t_1, \ldots, t_{q-1}, 1 \right) T^{-1}\right)},
		\end{equation}
		and the objective function $D^*: \real^q \to \real^{q-1}\times\bar{\real}$, where $\bar{\real}$ is the extended real line, is given by
		\begin{align*}
			D^*(t) := &\trans{\left( t_1,\ldots, t_{q-1},\inf_{x\in\mathcal{X}}\left[\trans{w(t)}\Gamma(x)\right] \right)}.
		\end{align*}
		
		The \emph{lower image} of \eqref{(D)} is $\mathcal{D} := D^*(\mathcal{T})-K$, where $\mathcal{T} := \{t\in \real^q \mid w(t)\in C^+\}$ is the feasible region of \eqref{(D)}. $\mathcal{D}$ is a closed convex set, see \cite{LRU14}.
		
		Let us now review the solution concept for the dual problem~\eqref{(D)} given in~\cite{Hey13, LRU14} and discuss the consequences it would have if the underlying primal problem~\eqref{eq:P} is unbounded.
		A finite $\varepsilon$-solution $\bar{\Tset} \subseteq \mathcal{T}$ of \eqref{(D)} is a finite set such that $D^*(t)$ is a $K$-maximal element of $D^*(\mathcal{T})$ for all $t\in\bar{\Tset}$ and 
		\begin{equation}\label{eq:dualsoln_old}
			\conv D^*(\bar{{\Tset}})-K+\varepsilon \{e^q\} \supseteq \mathcal{D}
		\end{equation}
		holds. Note that if problem \eqref{eq:P} is unbounded, then by Proposition \ref{prop: Pbounded} $\trans{w(t)}\Gamma(x)$ is unbounded from below, hence $D_q^*(t) = -\infty$ for some $t\in \mathcal{T}$. In particular, by Remark \ref{rem:extdirC+}, there exists $\bar{t}\in \mathcal{T}$ such that $w(\bar{t})$ is an extreme direction of $C^+$ and $D_q^*(\bar{t})= -\infty$. Moreover, if \eqref{eq:P} is not self-bounded, then by Lemma~\ref{lemma_weights} the set $W$ is not a closed set. In this case, the domain of problem \eqref{(D)} is not a closed set and it is not possible to find a finite $\varepsilon$-solution $\bar{\Tset}$ of \eqref{(D)} as \eqref{eq:dualsoln_old} can not be satisfied. Indeed, the projection of $\conv D^*(\bar{\Tset})-K$ onto its first $q-1$ components would not cover the projection of $\mathcal{D}$ onto its first $q-1$ components. 
		
		Motivated by this observation, we 
		introduce a generalized solution concept for problem~\eqref{(D)} that allows to treat also unbounded problems~\eqref{eq:P}. It is based on a $\delta$--outer approximation of $\mathcal{P}_{\infty}$ and an $\varepsilon$--solution of the geometric dual problem~\eqref{(D')} of the bounded problem~\eqref{eq:P'}.
		
		\begin{definition}{\label{solutionconcept_dual}}
			A pair $(\bar{\Tset},\mathcal{Y})$ is a $(\varepsilon, \delta)$--solution of~\eqref{(D)} if $\mathcal{Y}$ is a $\delta$--outer approximation of $\mathcal{P}_{\infty}$ and $\bar{\Tset}$ is an $\varepsilon$-solution of the modified  
			geometric dual problem
			\begin{align*}
				\label{(D')}
				\text{maximize~} D^*(t) \quad \text{with respect to~} \leq_K \quad \text{subject to~}  w(t) \in (\cone \mathcal{Y})^+.\tag{D'}
			\end{align*}
			
		\end{definition}
		
		Note that the modified 
		geometric dual problem~\eqref{(D')} is simply the geometric dual problem of~\eqref{eq:P'} defined in Remark~\ref{remark_definition}. 
		
		The following lemma will be used to prove the correctness of both the primal and the dual algorithm (presented below) with respect to dual solutions.
		
		\begin{lemma} \label{lemma:dual} 
			Let Assumption~\ref{assumption1} be satisfied. Let $\bar{\mathcal{X}}, \bar{\mathcal{T}}$ and  $\cone \mathcal{Y}_{Out}$ be respectively the set of weak minimizers, the set of maximizers and the outer approximation of $\mathcal{P}_\infty$ returned by Algorithm~\ref{alg_new}. Let $\mathcal{Y}$ in problem \eqref{(D')} be set to $\mathcal{Y}_{Out}$. \rev{Then,}
			\begin{enumerate}
				\item $\bar{\mathcal{T}}\supseteq \{\trans{T}y \mid y\text{~is an extreme direction of~} (\cone\mathcal{Y}_{Out})^+, c^\T y = 1\}$.
				\item $\bar{\mathcal{T}}$ consists of maximizers for problem \eqref{(D')}.
			\end{enumerate}
		\end{lemma}
		\begin{proof}
			\begin{enumerate}
				\item For each $x^{v,d} \in\bar{\mathcal{X}}$ (respectively for each $\frac{T^\T w^{v,d}}{c^\T w^{v,d}} \in \bar{\Tset}$), there exists $w^{v,d}$ with $\frac{T^\T w^{v,d}}{c^\T w^{v,d}} \in \bar{\Tset}$ (respectively $x^{v,d} \in\bar{\mathcal{X}}$) such that $$\{y\in \real^q \mid (w^{v,d})^\T( y-\Gamma(x^{v,d}))\geq0\} \rev{\supseteq \mathcal{P}^0}$$ \rev{by the construction of $\mathcal{P}^0$}. We abuse the notation here as some of these $w^{v,d}$'s are the generating vectors of $C^+$. The corresponding $x^{v,d}$ and $\frac{T^\T w^{v,d}}{c^\T w^{v,d}}$ are added to the respective sets in \rev{line 10 of Algorithm~\ref{alg_new}}. Let's denote the set of all such $w^{v,d}$'s by $\bar{\mathcal{W}}$. Then, we have $\bar{\Tset} = \{\frac{T^\T w}{c^\T w} \mid w \in \bar{\mathcal{W}} \}$ and $$\mathcal{P}_0 = \bigcap_{w^{v,d}\in \bar{\mathcal{W}}}\{y\in \real^q \mid (w^{v,d})^\T( y-\Gamma(x^{v,d}))\geq0\}.$$ This implies that $(\mathcal{P}_0)_\infty^+ = \cone \bar{\mathcal{W}}$ 
				and $(\mathcal{P}_0)_\infty = (\cone \bar{\mathcal{W}})^+$. Now, since $\mathcal{Y}_{Out} = \vertice \left( (\cone \bar{\mathcal{W}})^+ \cap B_1 (0) \right) \setminus \{0\}$, we have $\cone \mathcal{Y}_{Out} = (\cone \bar{\mathcal{W}})^+$, 
				hence $(\cone \mathcal{Y}_{Out})^+ = \cone \bar{\mathcal{W}}.$ Therefore, the set of extreme directions of $(\cone \mathcal{Y}_{Out})^+$ is a subset of $\bar{\mathcal{W}}$ and the result follows. 
				\item From \cite[Propositions 3.5 and 4.6]{LRU14}, $\bar{\Tset}$ returned by Algorithm \ref{alg_new} consists of maximizers for problem \eqref{(D)}. We will show that any $\bar{t}\in \bar{\Tset}$ is also feasible for $\eqref{(D')}$, and hence a maximizer of $\eqref{(D')}$. From the proof of the previous statement, any $\bar{t}\in \bar{\Tset}$ is of the form $\bar{t}=\frac{T^\T \bar{w}}{\bar{w}^\T c}$ for some $\bar{w}\in \bar{\mathcal{W}}$. Moreover, $\bar{t}_q = 1$ by construction of the matrix $T$. Then, for any $y\in \real^q$, we have $w(\bar{t})^\T y = \bar{t}^\T T^{-1}y = \frac{\bar{w}^\T y}{\bar{w}^\T c}$. Noting that $C \subseteq \cone \mathcal{Y}_{Out} = (\cone \bar{\mathcal{W}})^+$ and $c \in \Int C$, we have $\bar{w}^\T c > 0$. Moreover, $\bar{w}^\T y \geq 0$ for any $y\in \cone\mathcal{Y}_{Out}$. Thus, $w(\bar{t}) \in (\cone \mathcal{Y}_{Out})^+$. 
			\end{enumerate}
		\end{proof}
		
		We will show now that the primal Algorithm~\ref{alg_2} also solves the geometric dual problem~\eqref{(D)} in the sense of the solution concept for unbounded problems given in Definition \ref{solutionconcept_dual}.  
		
		\begin{theorem}
			\label{thm_alg_2_dual}
			Assume that the Algorithm~\ref{alg_2} terminated for problem~\eqref{eq:P} satisfying Assumption~\ref{assumption1}.
			Then, the pair $(\bar{\mathcal{T}}, \mathcal{Y}_{Out})$ outputted by Algorithm~\ref{alg_2} is a finite $(\varepsilon, \delta)$--solution of~\eqref{(D)}. 
		\end{theorem} 
		\begin{proof} 
			It is sufficient to show that the set $\bar{\mathcal{T}}$ returned by Algorithm \ref{alg_2} is a finite $\varepsilon$-solution to the modified geometric dual problem \eqref{(D')}, where the ordering cone $\mathcal{Y}$ is set to $\mathcal{Y}_{Out}$. Note that the only difference between \cite[Algorithm 1]{LRU14} applied for \eqref{(D')} and Algorithm \ref{alg_2} is in \rev{the initialization step:} the former initializes the set of maximizers as $$\{\trans{T}y \mid y\text{~is an extreme direction of~} (\cone\mathcal{Y}_{Out})^+, c^\T y = 1\},$$ whereas the latter may start with some additional maximizers for \eqref{(D')}, see Lemma~\ref{lemma:dual}. It follows from \cite[Theorem 4.9]{LRU14} that $\bar{\mathcal{T}}$ is a finite $\varepsilon$-solution to \eqref{(D')}. 
		\end{proof}

		Note that Algorithm~\ref{alg_2} simply solves problem \eqref{eq:P'} from Remark \ref{remark_definition} using the iterations of the primal algorithm from~\cite{LRU14}. The problem \eqref{eq:P'} can also be solved by applying the iterations of the geometric dual algorithm from~\cite{LRU14} after the same initialization phase conducted in Algorithm \ref{alg_2} and by setting the initial outer approximation $\mathcal{D}_0$ to the lower image of the geometric dual problem as 
		\begin{equation}\label{eq:D0} 
			\mathcal{D}_0 = \set{t\in\real^q \mid w(t)\in (\cone\mathcal{Y}_{Out})^+ } \cap \bigcap_{x^*\in \bar{{\X}}}\set{t\in \real^q \mid \trans{w(t)}\Gamma(x^*)-t_q \geq 0},\end{equation} 
		where $\bar{\X}$ and $\cone \mathcal{Y}_{Out}$ are respectively the set of (weak) minimizers and the outer approximation of $\mathcal{P}_\infty$ obtained from Algorithm~\ref{alg_new}, and $w: \real^q \to \real^q$ is defined as in~\eqref{eq:w(t)}. 
		
		
		Next, we provide the description of the geometric dual algorithm for solving a potentially unbounded problem~\eqref{eq:P}. 
		
		\begin{remark}\label{rem:condition_dual}
			The dual algorithm works with weighted sums scalarizations. To guarantee their boundedness we add an additional check -- see \rev{line}~2 of Algorithm~\ref{alg_dual} -- before proceeding with the iterations. If this test is not passed, we recommend to use the primal algorithm. To see from where the issue arises, recall that Lemma~\ref{lemma_weights} implies $(\cone \mathcal{Y}_{Out})^+ \subseteq \mathcal{P}_\infty^+ = \cl W$. This makes it possible for a weight $w \in \bd (\cone \mathcal{Y}_{Out})^+ \setminus W$ to exists for which the problem~\eqref{P1(w)} is unbounded. In \rev{line}~2 of Algorithm~\ref{alg_dual} we solve scalarizations \eqref{P1(w)} for all extreme directions $w$ of $(\cone \mathcal{Y}_{Out})^+$. If all of these problems are bounded, then $(\cone \mathcal{Y}_{Out})^+ \subseteq W$ is satisfied. Note that this is always the case for a self-bounded problem, see Lemma~\ref{lemma_weights}. 	
		\end{remark}

			{\scriptsize 
			\begin{center}
				\begin{minipage}{\linewidth}
					\begin{algorithm}[H]  
						\SetAlgoLined
									\caption{\label{alg_dual}\rev{Dual Algorithm for solving convex vector optimization problems}}
							\KwIn{ Problem~\eqref{eq:P} satisfying \rev{Assumption~\ref{assumption1}}, tolerances $\varepsilon, \delta > 0$.}  
							Apply Algorithm~\ref{alg_new} with tolerance $\delta$ to problem~\eqref{eq:P}: Compute set $\mathcal{Y}_{Out}$; set of weak minimizers $\bar{\X}$; set $\bar{\Tset}$; and $H$-representation of the initial outer approximation $\mathcal{D}_0$ \rev{of $\mathcal{D}$} as in \eqref{eq:D0}. \\
							Solve \eqref{P1(w)} for all extreme directions $w$ of $(\cone \mathcal{Y}_{Out})^+$. If all problems are bounded,  continue. Otherwise, use the primal algorithm. \\ 
							Set $k:=0$, $M := \emptyset$. \\
											\While {$M \neq \mathbb R^q$}{
							Set $M:=\real^q$ \\
							Compute the set $\mathcal{D}^V$ of vertices of $\mathcal{D}_k$\\
											\For{$v\in\mathcal{D}^V$}{
							Compute an optimal solution $x^{v}$ to Problem \eqref{P1(w)} where $w = w(v)$. \\
							$\bar{\X}\gets \bar{\X}\cup\set{x^v}$, $\bar{\Tset}\gets\bar{\Tset}\cup\set{\trans{T} w(v)}$. \\
											\If{ $v_q - \trans{w(v)}\Gamma(x^v)>\varepsilon$} {
											$M\gets M\cap\set{t \in \real^q \mid \trans{w(t)}\Gamma(x^v) \geq t_q}$. }
						}
											\If {$M\neq\real^q$} {
							$\mathcal{D}_{k+1} \gets \mathcal{D}_k\cap M$ and $k\gets k+1$.}
						}
											\KwOut{ Set of weak minimizers $\bar{\X}$, set of directions $\mathcal{Y}_{Out}$, set of maximizers $\bar{\mathcal{T}}$. }
						\end{algorithm}
					\end{minipage}
				\end{center}
			}
		
		\begin{theorem}\label{thm:dual_alg}
			Assume that Algorithm~\ref{alg_dual} terminated (outside of \rev{line}~2) for problem~\eqref{eq:P} satisfying Assumption~\ref{assumption1}. Then, the pairs $(\bar{\mathcal{X}}, \mathcal{Y}_{Out})$ and $(\bar{\mathcal{T}}, \mathcal{Y}_{Out})$ outputted by Algorithm~\ref{alg_dual} are a finite weak $(\varepsilon, \delta)$--solution of~\eqref{eq:P} and a finite $(\varepsilon, \delta)$--solution of~\eqref{(D)}, respectively.
		\end{theorem}
		\begin{proof}
			We will first show that $\mathcal{D}_0$ found in \rev{line 1 of Algorithm~\ref{alg_dual}} is an outer approximation of the lower image of~\eqref{(D')}. Note that $\set{t\in\real^q \mid w(t)\in (\cone\mathcal{Y}_{Out})^+ }$ is an intersection of vertical halfspaces \rev{(vertical in the sense that the last component of its normal direction is zero)} that contain the lower image of problem~\eqref{(D')}. Moreover, as argued in the proof of Lemma~\ref{lemma:dual}, each $x^*\in\bar{\mathcal{X}}$ corresponds to a $t\in\bar{\mathcal{T}}$ returned by Algorithm~\ref{alg_new}, which is a maximizer for problem \eqref{(D')}. Then, by~\cite[Proposition 4.13]{LRU14}, for each $x^\ast\in \bar{\X}$ the halfspace $\set{t\in \real^q \mid \trans{w(t)}\Gamma(x^*)-t_q \geq 0}$ is non-vertical and contains the lower image of~\eqref{(D')}. Moreover, by Lemma~\ref{lemma:dual}, $\bar{\mathcal{T}}$ found in \rev{line}~1 of Algorithm~\ref{alg_dual} consists of maximizers for \eqref{(D')}.
			If Algorithm~\ref{alg_dual} is continued after \rev{line} 2, then the remaining steps of \rev{Algorithm~\ref{alg_dual}} are the same as in~\cite[Algorithm 2]{LRU14}. In particular, \rev{lines 8-12} can be performed successfully (up to a near-optimal solution, see Remark~\ref{remark_near_opt}).
			Then, by~\cite[Theorem 4.14]{LRU14}, Algorithm \ref{alg_dual} returns a finite $\varepsilon$-solution $\bar{\mathcal{T}}$ to \eqref{(D')}; hence an $(\varepsilon,\delta)$-solution to \eqref{(D)}. Similarly, it returns a finite weak $\varepsilon$--solution $\bar{\X}$ of~\eqref{eq:P'}. Following the same steps from the proof of Theorem~\ref{thm_alg_2}, we conclude that $(\bar{\mathcal{X}}, \mathcal{Y}_{Out})$ is a finite weak $(\varepsilon, \delta)$--solution of~\eqref{eq:P}.
		\end{proof}

		\section{Examples}
		\label{sec_ex}
		In this section, we will consider three numerical examples. The first one is an illustrative example, where we will go through the algorithm step by step and depict and explain the intermediate steps in detail. The second example is coming from a financial application considered in\cite{RudUlu2021}. The involved problem is unbounded and was in~\cite{RudUlu2021} only solved by computing a candidate for the recession cone of the upper image. With help of the algorithms of the present paper the recession cone can be computed without additional arguments regarding the problem structure.
		The third example shows that the proposed algorithms can also be applied outside of the problem of solving CVOPs. It is used here to find polyhedral inner and outer $\delta$--approximations (in the sense of Definitions~\ref{solutionconcept} and ~\ref{solutionconcept2}) of a convex non-polyhedral cone. In the example we choose the ice cream cone.
		\begin{example}
			We illustrate the proposed Algorithms ~\ref{alg_new} and \ref{alg_2}
			on the following problem
			\begin{align}
				\label{prob_ex}
				\begin{split}
					&\text{minimize }  \trans{(x_1,x_2)} \quad \text{with respect to\ } \leq_{C} \\   
					&\text{ subject to } (x_1-1)^2 \leq x_2,
				\end{split}
			\end{align}
			where $C=\cone \{\trans{(1,0)},\trans{(1,2)}\}$. \rev{We fix $c = (\frac23,\frac13)^\T \in \Int C$.} The (exact) image of the feasible set $\Gamma(\mathcal{X})$ as well as the (exact) upper image $\mathcal{P}$ for this problem can be seen in Figure~\ref{exactupperimage}. The problem is neither bounded nor self-bounded and the (exact) recession cone of the upper image is $\mathcal{P}_{\infty} = \mathbb{R}^2_+$, which is a strict superset of the ordering cone $C$. This problem is simple enough to deduce these sets exactly. The algorithms proposed in this paper are used to compute inner and outer approximations of the in general unknown sets $\mathcal{P}_{\infty}$ and $\mathcal{P}$, as well as a weak $(\varepsilon,\delta)$--solution of problem~\eqref{prob_ex}.
			
			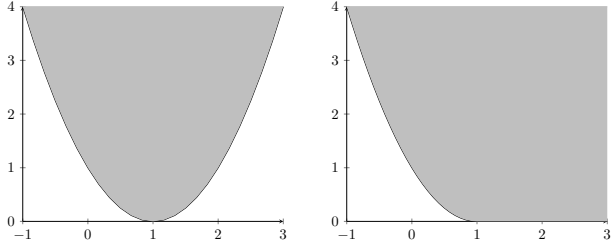
\begin{figure}[h]
				\centering
				\begin{tikzpicture}[scale=0.5]
					\begin{axis}[
						axis lines = left,
						]
						
						\addplot[name path=A,domain=-1:3,color=black]{(x-1)^2};
						\addplot[name path=C,domain=-1:3,color=lightgray]{4};
						\tikzfillbetween[of=A and C]{lightgray};
						
					\end{axis}
				\end{tikzpicture}
				\hspace{0.2cm}
				\begin{tikzpicture}[scale=0.5]
					\begin{axis}[
						axis lines = left,
						]
						
						\addplot[name path=A,domain=-1:1,color=black]{(x-1)^2};
						\addplot[name path=B,domain=1:3,color=black]{0};
						\addplot[name path=C,domain=-1:3,color=lightgray]{4};
						\tikzfillbetween[of=A and C]{lightgray};
						\tikzfillbetween[of=B and C]{lightgray};
						
					\end{axis}
				\end{tikzpicture}
				\caption{Image of the feasible set $\Gamma(\mathcal{X})$ (left) and upper image $\mathcal{P}$ (right) of Problem~\eqref{prob_ex}. \label{exactupperimage}}
			\end{figure}
			
			In the step by step illustration, we will focus on Algorithm~\ref{alg_new}, which corresponds to the first step of Algorithm~\ref{alg_2}, as the remaining steps of Algorithm~\ref{alg_2} are explained in detail in~\cite{LRU14} already.
			
			We apply Algorithm~\ref{alg_new} to Problem~\eqref{prob_ex} with a tolerance of  $\delta=0.1$.
			During the initialization, the algorithm verifies that Problem~\eqref{prob_ex} is feasible and unbounded and the initial point $v=(2.33,5.75)^\T$ in the interior of the upper image is computed. Furthermore three weak minimizers are found.
			The initial outer approximation $\mathcal{P}_0$ of the upper image, depicted in Figure~\ref{approxafterstep5}, is obtained through the supporting hyperplanes at these weak minimizers.
			
			\begin{figure}[h]
				\centering
				\subfloat[][]{\includegraphics[width=0.4\linewidth]{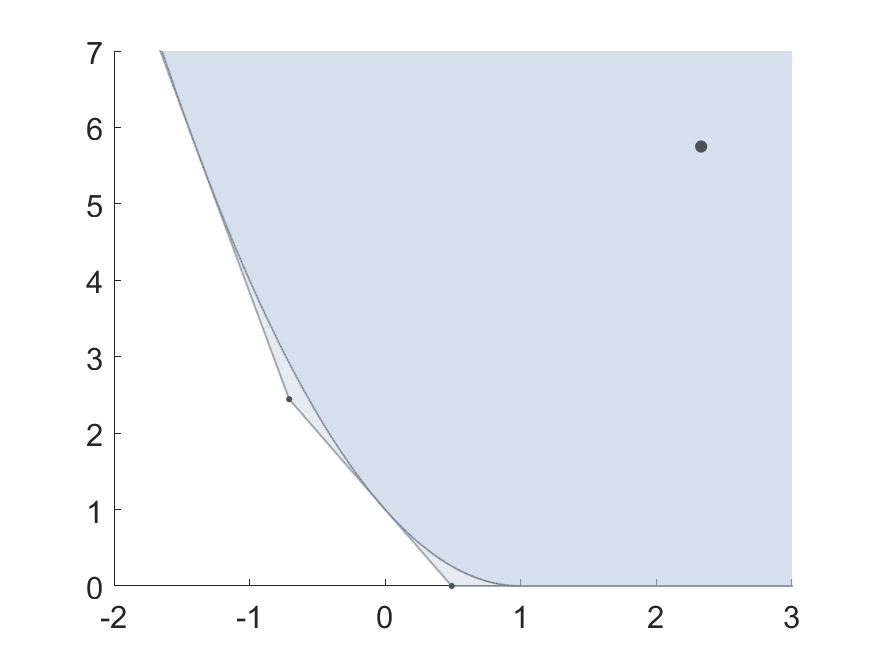}}%
				\qquad
				\subfloat[][]{\includegraphics[width=0.4\linewidth]{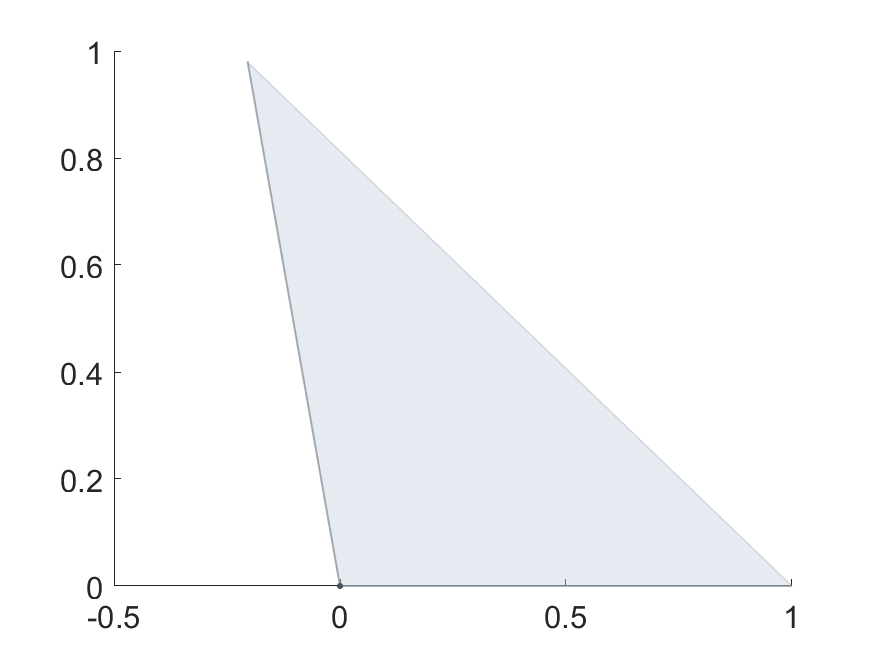}}%
				\caption{\rev{Upper image $\mathcal{P}$, its initial outer approximation $\mathcal{P}_0$ and} initial interior point $v$ (left) and the recession cone $(\mathcal{P}_0)_\infty$ (right). \label{approxafterstep5}}%
			\end{figure}
			
			Throughout the iterations of the algorithm, this initial outer approximation $\mathcal{P}_0$ \rev{of $\mathcal{P}$} is improved until its recession cone $(\mathcal{P}_0)_\infty$ is in $\delta$--distance to the recession cone $\mathcal{P}_\infty$ of the upper image  $\mathcal{P}$ in the sense of~\eqref{eq_def_1}.
			Thus, within the iterations of Algorithm~\ref{alg_new} we work with the intersection of the recession cone $(\mathcal{P}_0)_\infty$ of the  current outer approximation with the unit ball. The non-zero vertices of $(\mathcal{P}_0)_\infty \cap B_1(0)$ provide the set $\mathcal{Y}_{Out}$. These points are in each iteration compared with the set  of known inner directions $\mathcal{Y}_{In}$ (initially containing the normalized generating directions of the ordering cone $\mathcal{Y}_{In}=\{(1,0)^\T,(\frac{1}{3},\frac{2}{3})^\T\}$) and the set $\hat{C}$ (initially empty) of known outer direction with at most $\delta$ distance to the true recession cone $\mathcal{P}_\infty$. All of these sets at the start of the iterations are depicted in Figure~\ref{setsbeforeloop}. How these sets change over the iterations can be seen in Figures~\ref{iterations123} and~\ref{iterations456}.
			Table~\ref{table_iter} summarizes the iterations further.
			
			\begin{figure}[h]
				\centering
				{\includegraphics[width=0.35\linewidth]{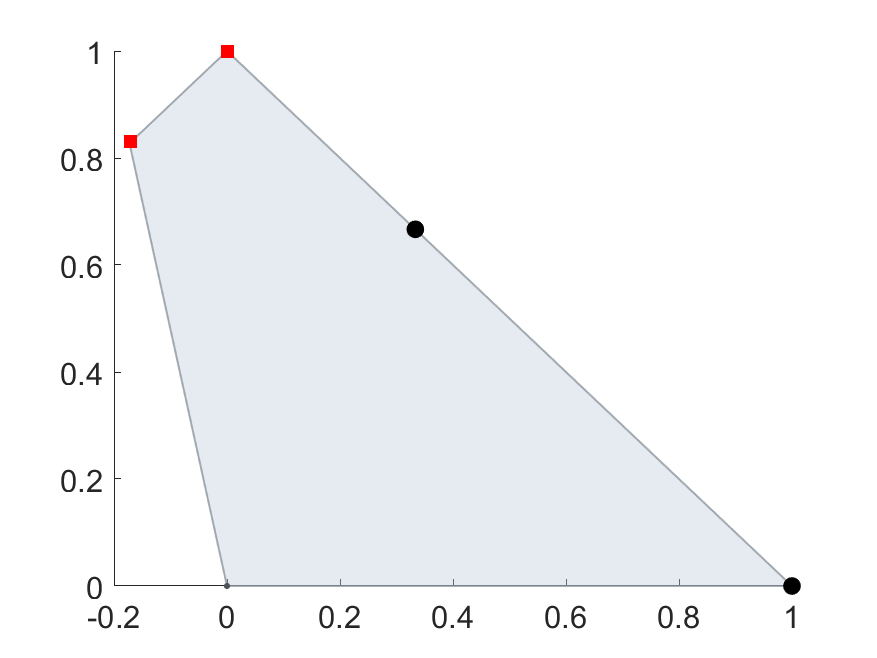}}%
				\caption{Set $(\mathcal{P}_0)_\infty \cap B_1(0)$ after \rev{the initialization}. 
					Highlighted are elements of  $\mathcal{Y}_{In}$ (\textcolor{black}{$\bullet$}) and  $\mathcal{Y}_{Out} \setminus (\mathcal{Y}_{In} \cup \hat{C})$ ({$\csquare{red}$}), where $\hat{C}$ is still the empty set. 
					\label{setsbeforeloop}}
			\end{figure}

			\begin{figure}[h]
				\centering
				\subfloat[][]{\includegraphics[width=0.28\linewidth]{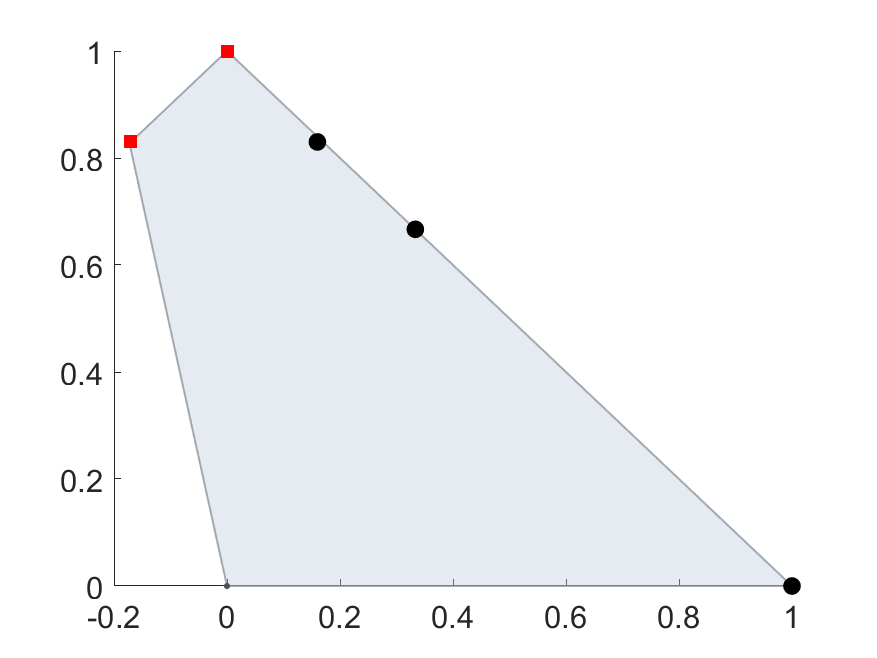}}%
				\qquad
				\subfloat[][]{\includegraphics[width=0.28\linewidth]{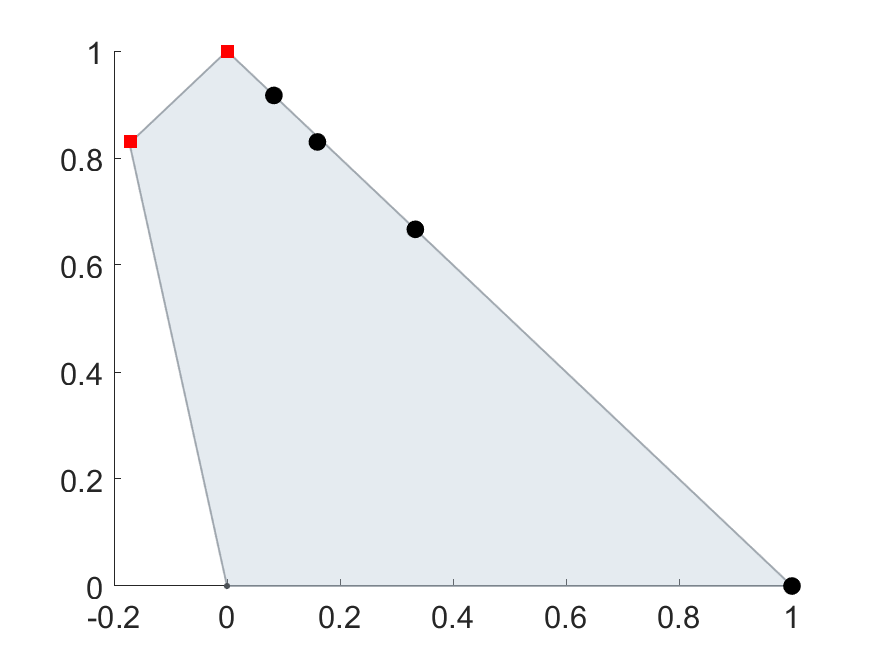}}%
				\qquad
				\subfloat[][]{\includegraphics[width=0.28\linewidth]{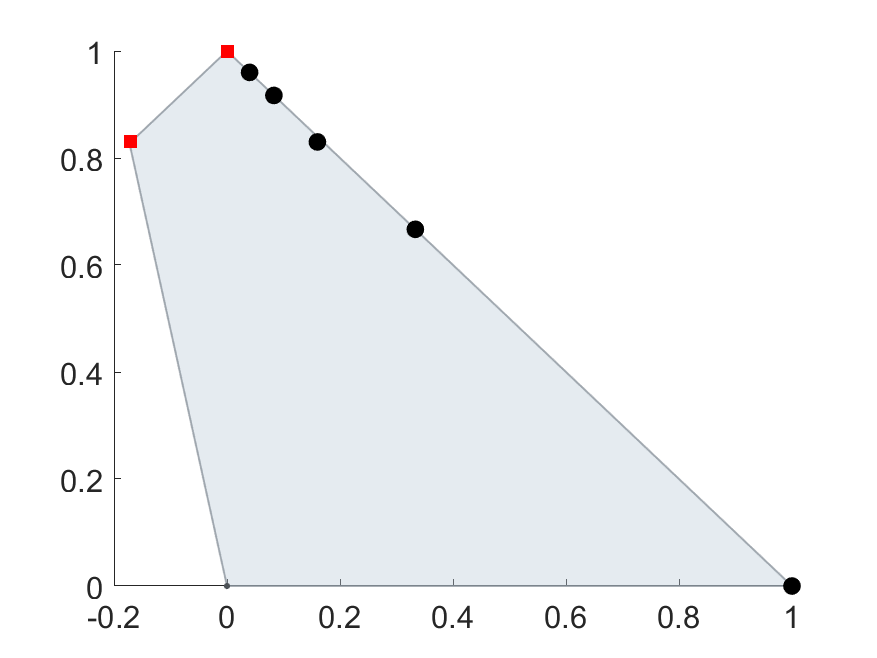}}%
				\caption{Set $(\mathcal{P}_0)_\infty \cap B_1(0)$ with $\mathcal{Y}_{In}$ (\textcolor{black}{$\bullet$}) and  $\mathcal{Y}_{Out} \setminus (\mathcal{Y}_{In} \cup \hat{C})$ ({$\csquare{red}$})  in Iterations~1-3 of \rev{Algorithm~\ref{alg_new}}.
					$\hat{C}$ is still the empty set. 
					\label{iterations123}}
			\end{figure}
			
			\begin{figure}[h]
				\centering
				\subfloat[][]{\includegraphics[width=0.28\linewidth]{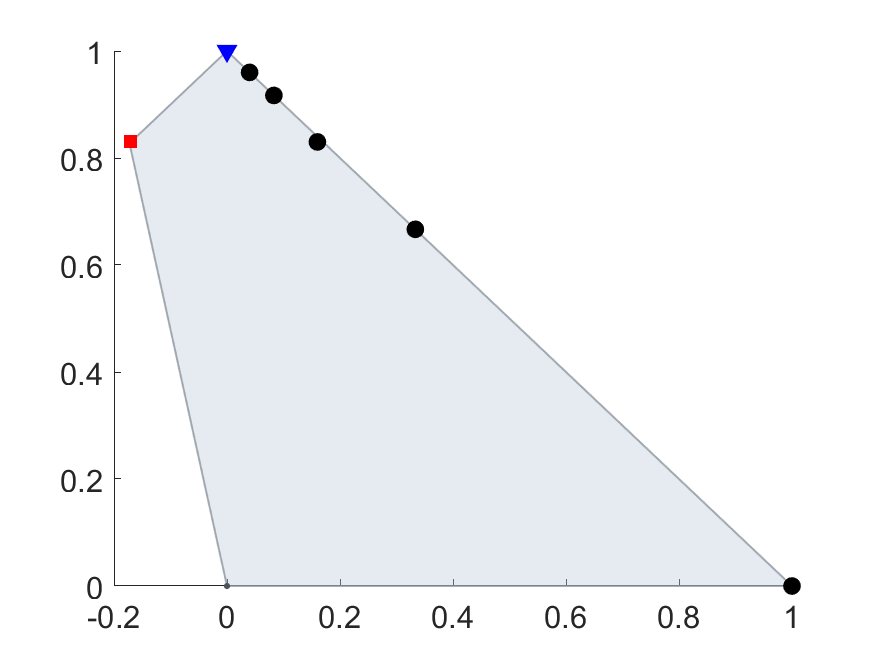}}%
				\qquad
				\subfloat[][]{\includegraphics[width=0.28\linewidth]{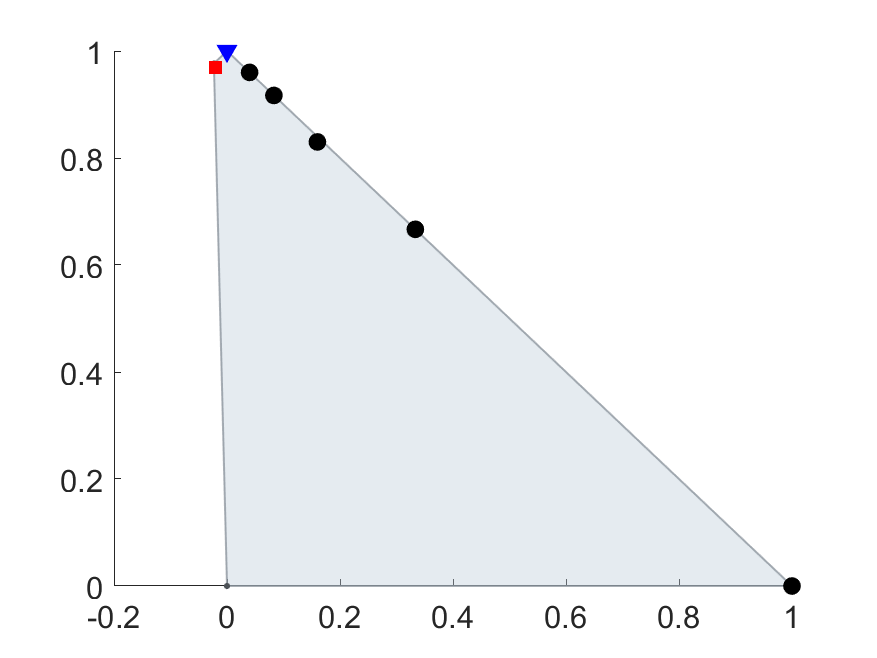}}%
				\qquad
				\subfloat[][]{\includegraphics[width=0.28\linewidth]{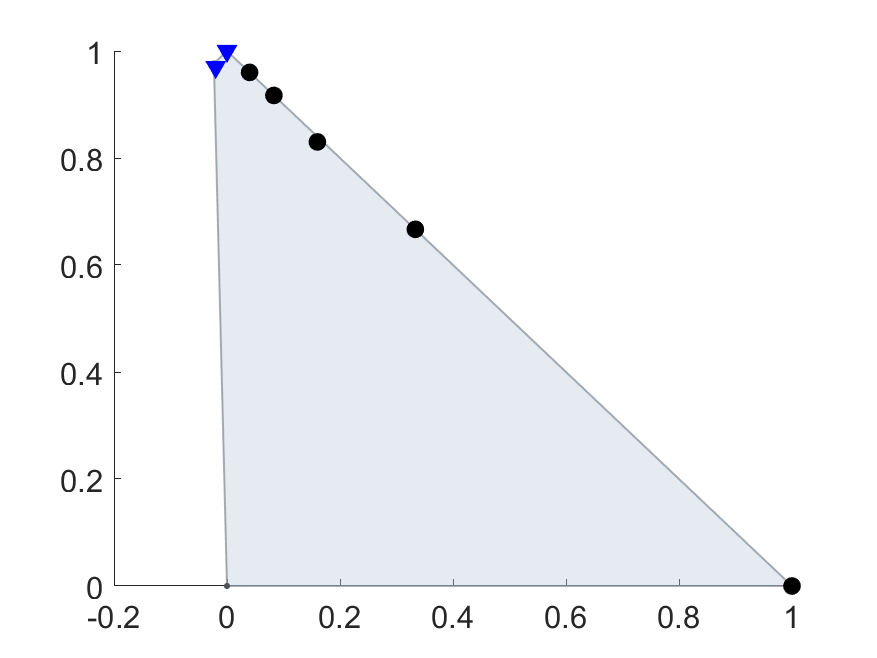}}%
				\caption{Set $(\mathcal{P}_0)_\infty \cap B_1(0)$ with $\mathcal{Y}_{In}$ (\textcolor{black}{$\bullet$}), $\hat{C}$ ({$\ctriangle{blue}$}) and  $\mathcal{Y}_{Out} \setminus (\mathcal{Y}_{In} \cup \hat{C})$ ({$\csquare{red}$}) in Iterations~4-6 of \rev{Algorithm~\ref{alg_new}}. \label{iterations456}}
			\end{figure}

	\begin{table}[h]
		\centering
		\resizebox{0.95\textwidth}{!}{
			\begin{tabular}{|c|c|c|c|c|}
				\hline
				Iteration & $d$ & $\mathcal{Y}_{In}$ & $\hat{C}$ & $\mathcal{Y}_{Out} \setminus (\mathcal{Y}_{In} \cup \hat{C})$\\
				\hline
				1 & $( 0, 1 )^\T$ & $\mathcal{Y}_{In} \cup \{(0.16, 0.83)^\T\}$ & $\emptyset$ & $\left\{(0 , 1)^\T,( -0.17, 0.83 )^\T \right\}$ \\
				\hline
				2 & $( 0, 1 )^\T$ & $\mathcal{Y}_{In} \cup \{(0.083, 0.917)^\T\}$ & $\emptyset$ & $\left\{(0 , 1)^\T,( -0.17, 0.83 )^\T \right\}$ \\
				\hline
				3 & $( 0, 1 )^\T$ & $\mathcal{Y}_{In} \cup \{(0.04, 0.96)^\T\}$ & $\emptyset$ & $\left\{(0 , 1)^\T( -0.17, 0.83 )^\T \right\}$ \\
				\hline
				4 & $( 0, 1 )^\T$ & $\mathcal{Y}_{In}$ & $\{(0, 1)^\T\}$ & $\left\{( -0.17, 0.83 )^\T \right\}$ \\
				\hline
				5 & $( -0.17, 0.83 )^\T$ & $\mathcal{Y}_{In}$ & $\hat{C}$ & $\left\{( -0.02, 0.97 )^\T \right\}$ \\
				\hline
				6 & $( -0.02, 0.97 )^\T$ & $\mathcal{Y}_{In}$ & $\hat{C}\cup \{( -0.02, 0.97 )^\T\}$ & $\emptyset$ \\
				\hline
		\end{tabular}}
		\caption{\label{table_iter} The sets $\mathcal{Y}_{In}, \hat{C}$ and $\mathcal{Y}_{Out} \setminus (\mathcal{Y}_{In} \cup \hat{C})$ after each iteration \rev{of Algorithm~\ref{alg_new}} as well as the considered direction $d$ in each iteration. If a set has not changed after an iteration we write the set itself.}
	\end{table}

\end{example}

	Let us now consider Table~\ref{table_iter} in detail:
	In the first three iterations the selected $d \in \mathcal{Y}_{Out}$ has a distance greater than $\delta$ to the nearest point $\tilde{r}$ of the inner direction set, so one solves the Pascoletti-Serafini scalarization \rev{(P$_2 (v, \frac{ d + \tilde{r}}{\norm{ d + \tilde{r}}})$)}. 
	Since the scalarization problem is unbounded, 
	the set of inner directions is updated \rev{(line 24)}. In Iteration~4, the selected $d$ has a distance less or equal than $\delta$ to $\mathcal{Y}_{In}$, so 
	$\hat{C}$ \rev{is updated (line 20)}. In Iteration~5 the selected $d$ has a distance greater than $\delta$ to the nearest point $\tilde{r}$ of the inner set, so again the Pascoletti-Serafini scalarization for a convex combination of $d$ and $\tilde{r}$ is solved. The scalarization yields an optimal solution. 
	Hence one computes an updated polytope which yields a new set of directions $\mathcal{Y}_{Out}$ \rev{(lines 26-27)}. Finally, in Iteration~6, 
	\rev{the selected $d$ has a distance less or equal than $\delta$}, which ends the loop and terminates the algorithm since the set of outer directions with distance greater than $\delta$ to the inner set is now empty.

	Algorithm~\ref{alg_new}  outputs in particular the final
	initial outer approximation $\mathcal{P}_0$ \rev{of $\mathcal{P}$}, the set\\ $\mathcal{Y}_{In}=\{(1,0)^\T,(\frac{1}{3},\frac{2}{3})^\T,(0.16,0.83)^\T,(0.083,0.917)^\T,(0.04,0.96)^\T\}$, as well as the set\\
	$\mathcal{Y}_{Out} = \{(1,0)^\T,(0,1)^\T, (-0.02,0.97)^\T\}$. Note that some directions of $\mathcal{Y}_{In}, \mathcal{Y}_{Out}$ can be omitted. By 
	$$\cone \mathcal{Y}_{In}=\cone\{(1,0)^\T,(0.04,0.96)^\T\}, \cone \mathcal{Y}_{Out}=\cone \{(-0.02,0.97)^\T,(1,0)^\T\},$$ both depicted in Figure~\ref{inneroutercone}, a finite $\delta$--inner-- and a finite $\delta$--outer approximation of $\mathcal{P}_\infty$ is provided. 
	
	\begin{figure}[h]
		\centering
		\includegraphics[width=0.38\linewidth]{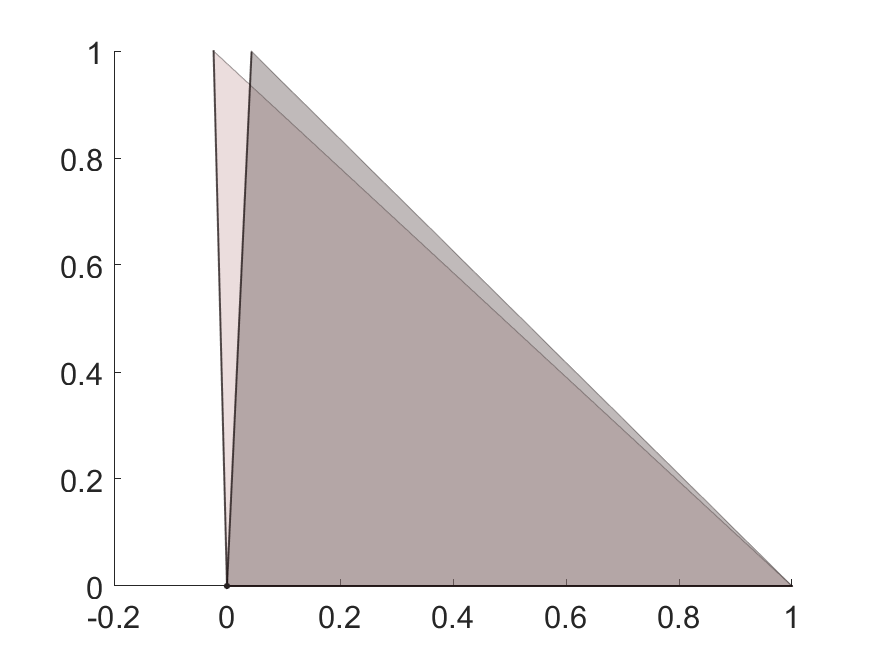}
		\caption{Sets $\cone \mathcal{Y}_{In}$ and $\cone \mathcal{Y}_{Out}$.} 
		\label{inneroutercone}
	\end{figure}
	
	Algorithm~\ref{alg_2} continues using $\cone \mathcal{Y}_{Out}$ as the new ordering cone, that is, it solves the now bounded problem~\eqref{eq:P'} using the algorithm from \cite{LRU14}. One thus obtains a weak $\varepsilon$--solution $\bar{\X}$ of~\eqref{eq:P'}. Then, the pair $(\bar{\X}, \mathcal{Y}_{Out})$ provides a weak $(\varepsilon,\delta)$--solution of problem~\eqref{prob_ex}. The obtained outer- and inner approximation of the upper image is depicted in Figure~\ref{upperimage}.

	\begin{figure}[h]
		\centering
		\subfloat[][]{\includegraphics[width=0.4\linewidth]{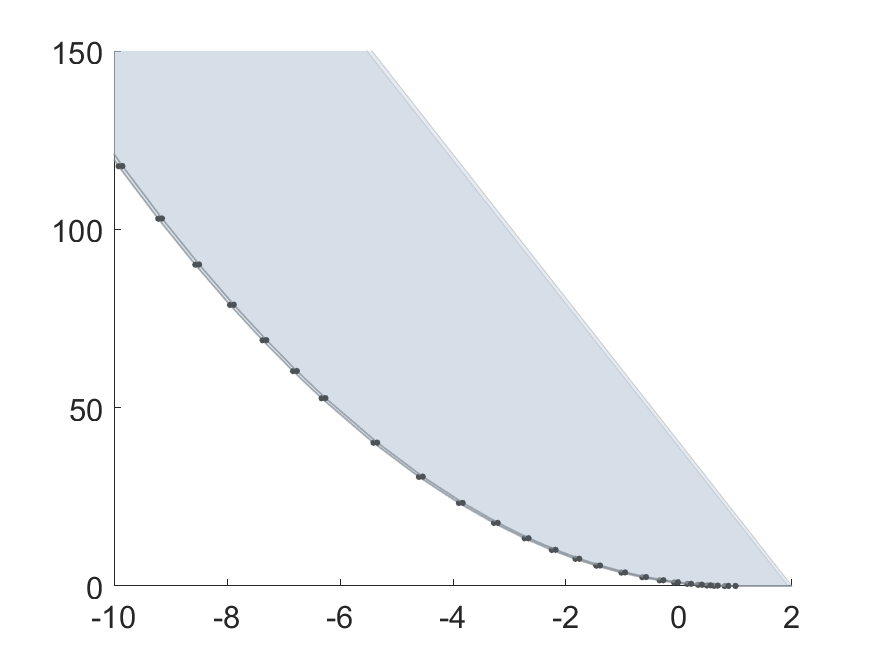}}%
		\qquad
		\subfloat[][]{\includegraphics[width=0.4\linewidth]{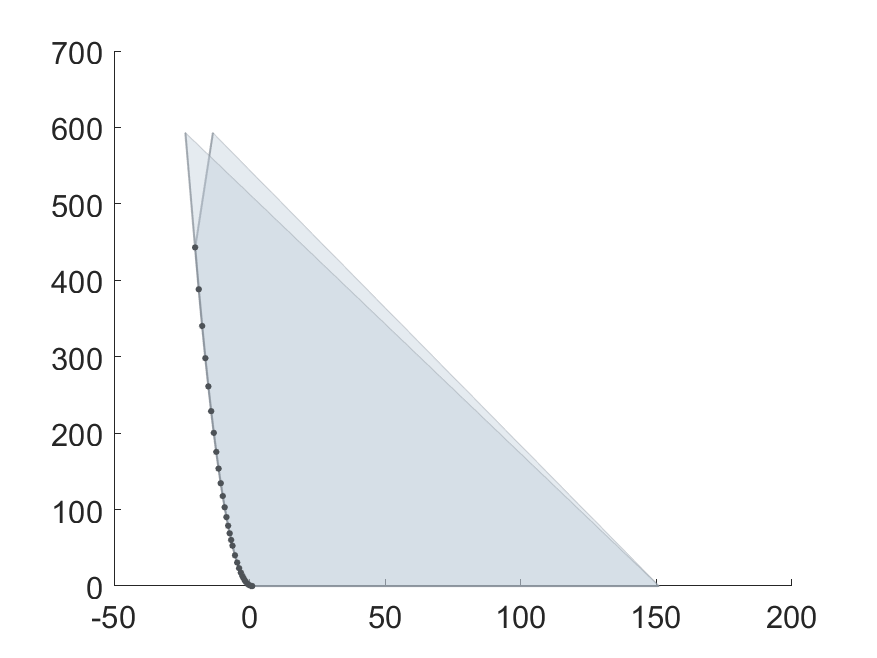}}%
		\caption{Inner- and outer approximation of the upper image, zoomed in (left) and zoomed out (right).}
		\label{upperimage}
	\end{figure}

	\begin{example} \label{ex:util}
		In \cite{RudUlu2021}, convex vector optimization problems are solved to compute set-valued indifference buy and sell prices where the underlying preference relation and/or the financial market are not complete. In particular, in \cite[Example 6.7]{RudUlu2021}, a conical market model is considered where the initial solvency cone $\real^2_+ \subsetneq K_0 \subsetneq \real^2$ is given by $K_0 = \cone\{\trans{(1,-0.9)},\trans{(-0.9,1)}\}$. The ordering cone for the two CVOPs to be solved (given by \cite[equations (21),(22)]{RudUlu2021}) in order to compute the set valued buy and sell prices is $\real^2_+$. However, it is shown in \cite[Section 6.2]{RudUlu2021} that these CVOPs are unbounded with respect to $\real^2_+$. To overcome this difficulty, it is shown analytically that the recession cones of the upper images are at least as large as the cone $K_0$. Using this observation, in \cite[Example 6.7]{RudUlu2021}, the ordering cone of the problems are set to $K_0$, \cite[Algorithm 1]{LRU14} is called for each problem and since the algorithm returns solutions to these CVOPs, it is concluded that the problems are bounded with respect to $K_0$. 
		
		We consider these two CVOPs with their original ordering cone, $\real^2_+$, and run Algorithm \ref{alg_new} with $\delta = 0.01$. For both problems we obtain $\mathcal{Y}_{Out} = \{\trans{(1,-0.9)},\trans{(-0.9,1)}\}$ and $\mathcal{Y}_{In} = \{\trans{(1,-0.8867)},\trans{(-0.8867,1)}\}$. As $\cone \mathcal{Y}_{Out}=K_0$, we can by \cite{RudUlu2021} conclude that the outer approximation of the recession cone returned by Algorithm \ref{alg_new} coincides with the true recession cone. 
	\end{example}

	\begin{example}
		\label{ex_ice}
		The ice cream cone
			$C_{ice} = \left\lbrace x \in \mathbb{R}^3 \mid \sqrt{x_1^2 + x_2^2} \leq x_3 \right\rbrace $
			is a convex cone having all the properties we require of an ordering cone within this paper aside from not being polyhedral. We used our algorithm to find a polyhedral (inner and outer) approximation of the ice cream cone. To do this we solve the problem
			\begin{align}
				\label{prob_ice}
				\text{minimize~} x \quad \text{with respect to~} \leq_{C_i} \quad \text{subject to~} x \in C_{ice},
			\end{align}
			where $C_i$ is a polyhedral ordering cone contained within the ice cream cone $C_{ice}$. Since the upper image of this problem is a cone, namely the ice cream cone, problem~\eqref{prob_ice} can be solved via Algorithm~\ref{alg_new}, the second phase of Algorithm~\ref{alg_2} is not needed in this case. We solve problem~\eqref{prob_ice} for two choices of ordering cones, $C_1 = \cone \{ (1, 0, 1)^\T, (-1, 0, 1)^\T,(0, 1, 1)^\T, (0, -1, 1)^\T \}$ and $C_2 = \cone \{ (1, 0, 1)^\T, (0, 1, 1)^\T,$ $ (0, 0, 1)^\T\}$, and a tolerance of $\delta = 0.2$. The resulting approximations are visualized in Figures~\ref{fig_ice1} and ~\ref{fig_ice2}. 
			
			\begin{figure}[h]
				\subfloat[][]{\includegraphics[width = 0.45\textwidth]{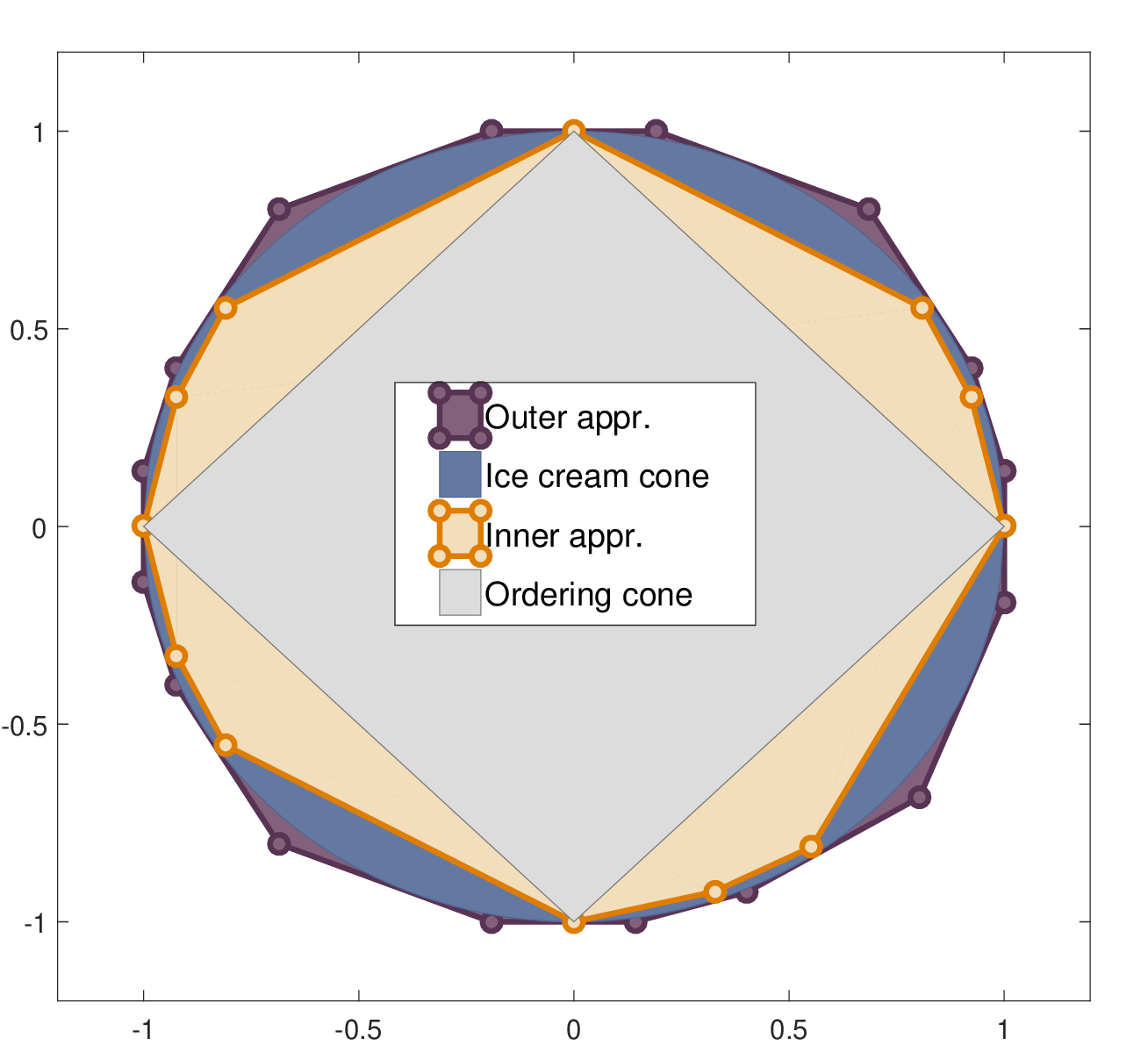}}
				\qquad
				\subfloat[][]{\includegraphics[width = 0.45\textwidth]{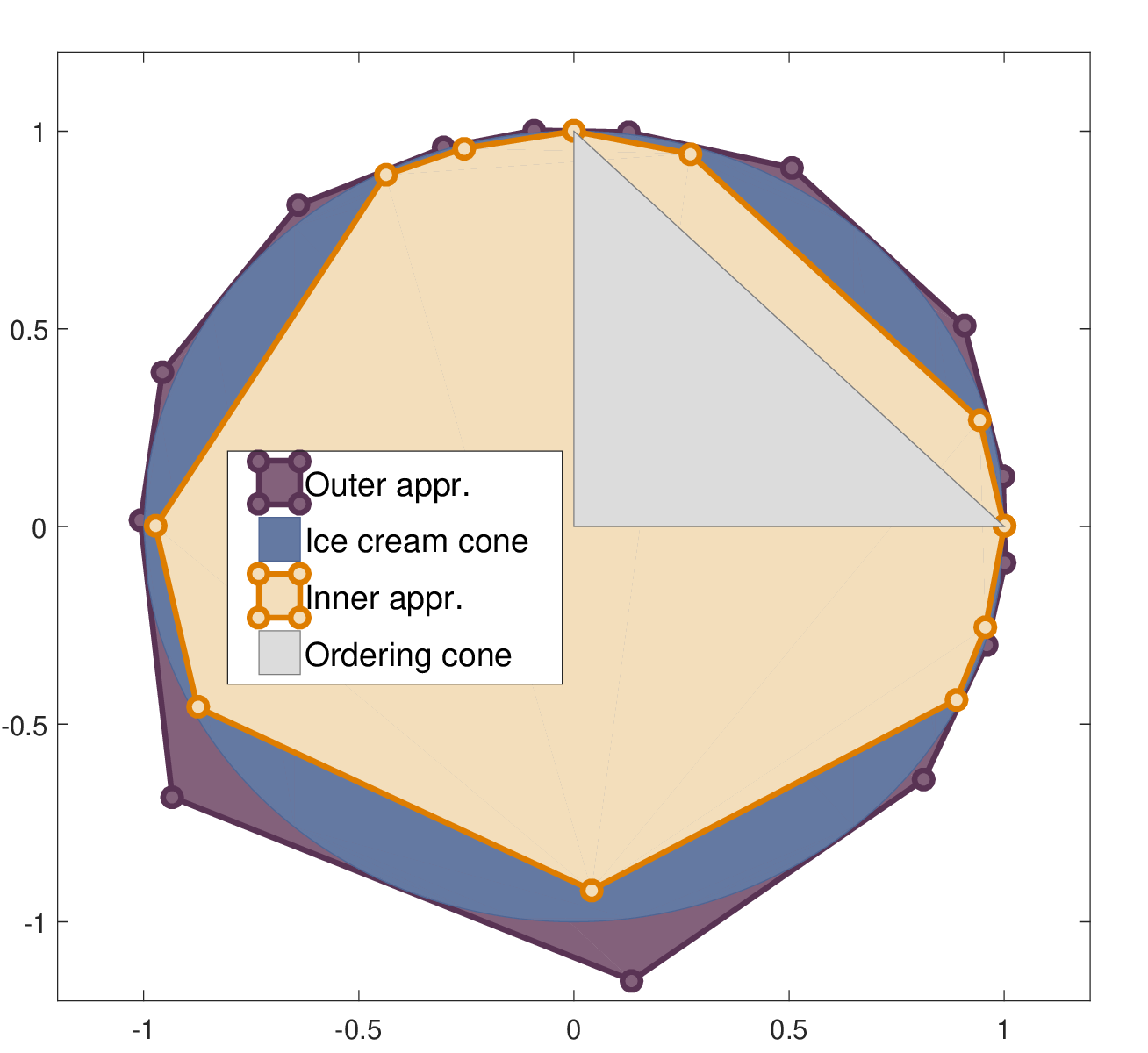}}
				\caption{\label{fig_ice1} On an intersection with a plane $\{ x \in \mathbb{R}^3 \mid x_3 = 1\}$  we depict the ice cream cone $C_{ice}$, the inner and outer approximations \rev{of it} obtained from Algorithm~\ref{alg_new} and the used ordering cone. The left side corresponds to ordering cone $C_1$, the right side to ordering cone $C_2$.}
			\end{figure}
			\begin{figure}[h]
				\subfloat[][]{\includegraphics[width = 0.45\textwidth]{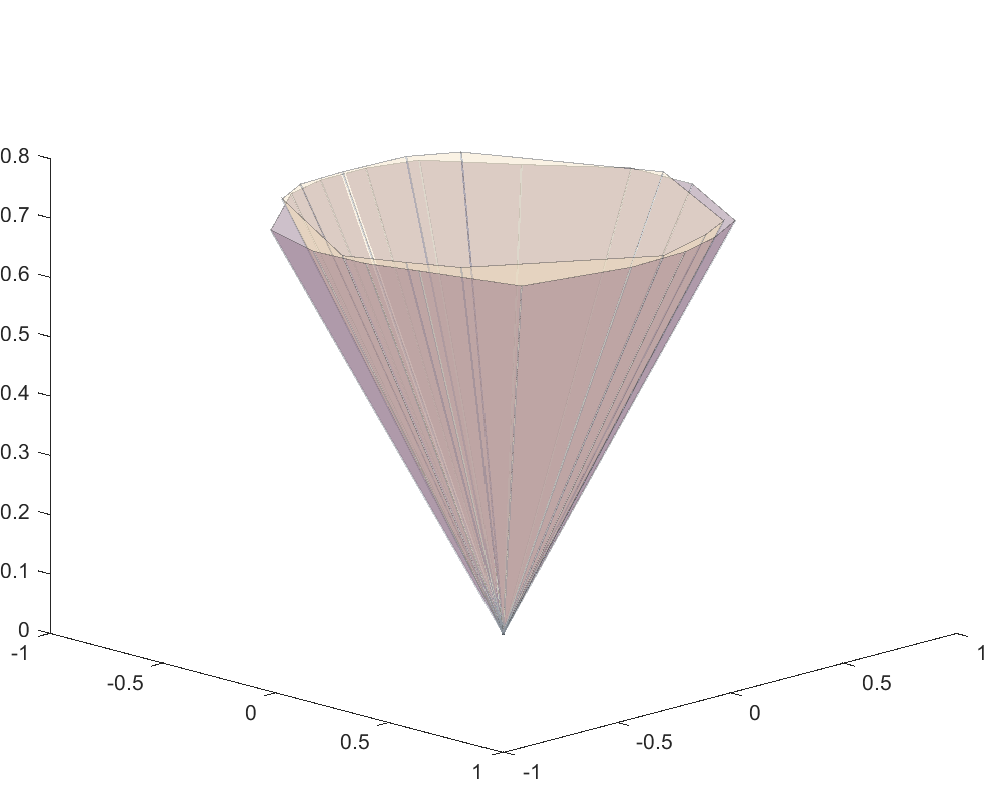}}
				\qquad
				\subfloat[][]{\includegraphics[width = 0.45\textwidth]{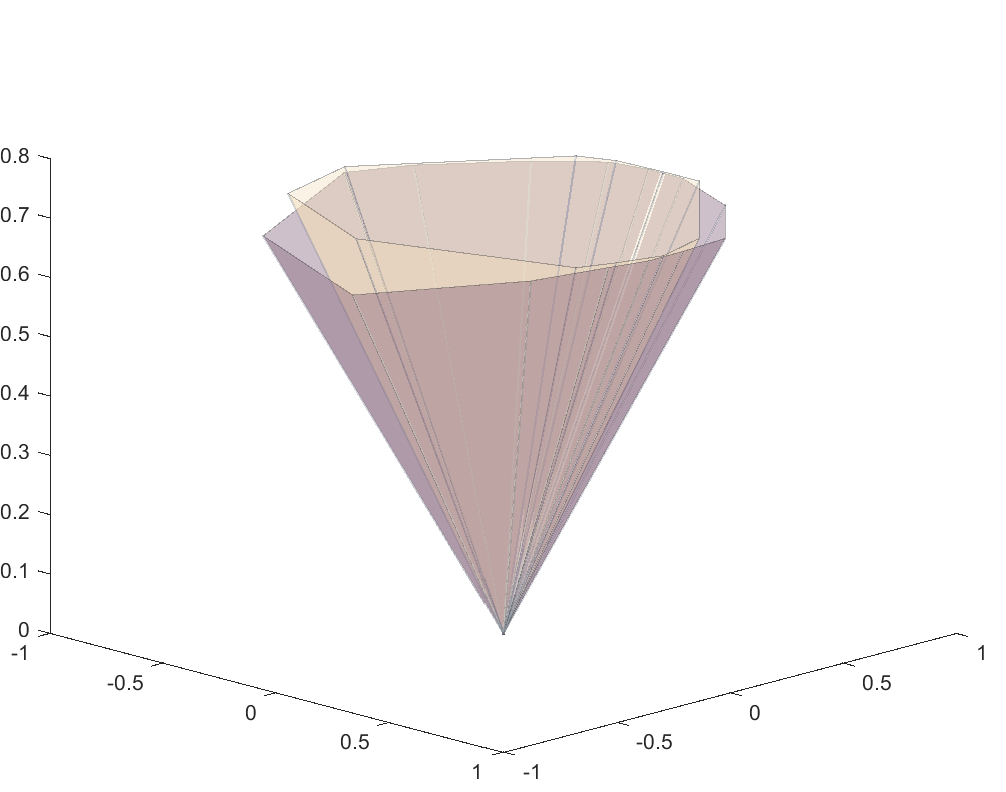}}
				\caption{\label{fig_ice2} Inner and outer approximations of the ice cream cone obtained from Algorithm~\ref{alg_new}. The left side corresponds to ordering cone $C_1$, the right side to ordering cone $C_2$.}
			\end{figure}
			
		\end{example}

		\section{Conclusion} \label{sec:conc}
		\rev{
			When it comes to solving convex vector optimization problems, the prior literature provides a solution concept and methods appropriate for bounded problems, but not for unbounded ones. In this paper we consider the unbounded case: Firstly, we propose a generalized solution concept appropriate for a convex vector optimization problem regardless of whether it is bounded or not. Secondly, we provide an algorithm for computing directions that approximate the recession cone of the upper image of~\eqref{eq:P}. And thirdly, we combine this algorithm with known methods, both primal and dual, to compute a solution, primal as well as dual, of the convex vector optimization problem.
		}
		
		\rev{
			The results closest to those obtained here can be found in the works of~\cite{Doerfler22, DorLoh2022}. The authors of~\cite{Doerfler22, DorLoh2022} do not solve convex vector optimization problems, but rather search for polyhedral approximations to convex sets, in particular so-called spectahedrons and spectahedral shadows. Our $(\epsilon, \delta)$-solution is closely related to the concept of $(\epsilon, \delta)$-polyhedral approximation of a convex set proposed in \cite{Doerfler22}. Independently of our work, \cite{DorLoh2022} introduce an algorithm for approximating recession directions of spectahedral shadows. 
		}
		
		\rev{
			Since our work is the first attempt to handle unbounded convex vector optimization problems in the literature, there are still open questions to be answered and generalizations to be made. An important question concerns the ordering cone. Here, as well as in the literature handling the bounded case, a polyhedral ordering cone is assumed for computational reasons. But what about problems with non-polyhedral ordering cones? One possibility could be to first approximate the non-polyhedral ordering cone through a polyhedral one (using Algorithm~\ref{alg_new} as in Example~\ref{ex_ice}) and then use the approximation to solve (a modified) vector optimization problem. The difficulty of this approach lies with bounding the approximation error for the directions. We leave this question for future research.
		}
		
		\rev{
			Remark~\ref{rem:lineality} suggests further questions that can be explored with regards to Algorithm~\ref{alg_new} for approximating recession directions: The current version works with arbitrary direction selection (similarly to arbitrary vertex selection of~\cite[Algorithm 1]{LRU14}), different selection rules could be examined in the future (similarly to vertex selection rules of~\cite{DorLohSchWei2021,KesUlus2022}). Another question to explore with regard to the direction is under what assumptions (and at what costs) could also directions be obtained in the pre-image space.  
		}
		
		\rev{
			Furthermore, the dual Algorithm~\ref{alg_dual} works under an additional assumption which guarantees that all relevant weights $w$ lead to a bounded weighted sum scalarization problem. One could explore the idea of disturbing the problematic weights $w$ slightly to guarantee boundedness. The important question is how to chose such 'slight disturbance' and how to translate it into the error tolerance.
		}

	\section*{Acknowledgments}
F. Ulus and B. Rudloff acknowledge support from the OeNB anniversary fund, project number 17793.

\bibliographystyle{plain}
\bibliography{database+}

\end{document}